\documentclass[10pt, a4paper, reqno]{amsart}

\usepackage{tikz-cd}
\usepackage{array}
\usetikzlibrary{snakes}

\usepackage{amsmath, amssymb, amsfonts, amstext, amsthm, amscd}
\usepackage[mathscr]{euscript}
\usepackage{enumerate}

\usepackage{float}

\usepackage[raiselinks=false,colorlinks=true,citecolor=blue,urlcolor=blue,linkcolor=blue,bookmarksopen=true,pdftex]{hyperref} 

\usepackage[paperheight=11.2in,paperwidth=8.1in,bindingoffset=0in,left=0.9in,right=0.9in,top=1in,bottom=1in,headsep=.5\baselineskip]{geometry}
\usepackage{tikz}
\usepackage{tikz-cd}
\usetikzlibrary{snakes}
\usetikzlibrary{intersections, calc}
\usetikzlibrary{decorations.markings}
\usetikzlibrary{arrows.meta}
\usetikzlibrary{bending}

\usepackage[english]{babel}
\usepackage{chngcntr}

\usepackage{pdflscape} 

\begin{document}
	
	\thispagestyle{empty} 
	\def\theequation{\arabic{section}.\arabic{equation}}

	\newcommand{\codim}{\mbox{{\rm codim}$\,$}}
	\newcommand{\stab}{\mbox{{\rm stab}$\,$}}
	\newcommand{\lr}{\mbox{$\longrightarrow$}}
	
	\newcommand{\be}{\begin{equation}}
		\newcommand{\ee}{\end{equation}}
	
	\newtheorem{guess}{Theorem}[section]
	\newcommand{\bth}{\begin{guess}$\!\!\!${\bf }~}
		\newcommand{\eeth}{\end{guess}}
	\renewcommand{\bar}{\overline}
	\newtheorem{propo}[guess]{Proposition}
	\newcommand{\bpropo}{\begin{propo}$\!\!\!${\bf }~}
		\newcommand{\epropo}{\end{propo}}
	
	\newtheorem{lema}[guess]{Lemma}
	\newcommand{\blem}{\begin{lema}$\!\!\!${\bf }~}
		\newcommand{\elem}{\end{lema}}
	
	\newtheorem{defe}[guess]{Definition}
	\newcommand{\bdefe}{\begin{defe}$\!\!\!${\bf }~}
		\newcommand{\edefe}{\end{defe}}
	
	\newtheorem{coro}[guess]{Corollary}
	\newcommand{\bcor}{\begin{coro}$\!\!\!${\bf }~}
		\newcommand{\ecor}{\end{coro}}
	
	\newtheorem{rema}[guess]{Remark}
	\newcommand{\brem}{\begin{rema}$\!\!\!${\bf }~\rm}
		\newcommand{\erem}{\end{rema}}
	
	\newtheorem{exam}[guess]{Example}
	\newcommand{\beg}{\begin{exam}$\!\!\!${\bf }~\rm}
		\newcommand{\eeg}{\end{exam}}
	
	\newtheorem{notn}[guess]{Notation}
	\newcommand{\bnot}{\begin{notn}$\!\!\!${\bf }~\rm}
		\newcommand{\enot}{\end{notn}}

	\newcommand{\ch}{{\mathcal H}}
	\newcommand{\cf}{{\mathcal F}}
	\newcommand{\cd}{{\mathcal D}}
	\newcommand{\cR}{{\mathcal R}}
	\newcommand{\cv}{{\mathcal V}}
	\newcommand{\cn}{{\mathcal N}}
	\newcommand{\lra}{\longrightarrow}
	\newcommand{\ra}{\rightarrow}
	\newcommand{\blr}{\Big \longrightarrow}
	\newcommand{\da}{\Big \downarrow}
	\newcommand{\ua}{\Big \uparrow}
	\newcommand{\hra}{\mbox{{$\hookrightarrow$}}}
	\newcommand{\rt}{\mbox{\Large{$\rightarrowtail$}}}
	\newcommand{\dua}{\begin{array}[t]{c}
			\Big\uparrow \\ [-4mm]
			\scriptscriptstyle \wedge \end{array}}
	\newcommand{\ctext}[1]{\makebox(0,0){#1}}
	\setlength{\unitlength}{0.1mm}
	\newcommand{\cl}{{\mathcal L}}
	\newcommand{\cp}{{\mathcal P}}
	\newcommand{\ci}{{\mathcal I}}
	\newcommand{\bz}{\mathbb{Z}}
	\newcommand{\cs}{{\mathcal s}}
	\newcommand{\ce}{{\mathcal E}}
	\newcommand{\ck}{{\mathcal K}}
	\newcommand{\cz}{{\mathcal Z}}
	\newcommand{\cg}{{\mathcal G}}
	\newcommand{\ct}{{\mathcal T}}
	\newcommand{\cj}{{\mathcal J}}
	\newcommand{\cc}{{\mathcal C}}
	\newcommand{\ca}{{\mathcal A}}
	\newcommand{\cb}{{\mathcal B}}
	\newcommand{\cx}{{\mathcal X}}
	\newcommand{\co}{{\mathcal O}}
	\newcommand{\bq}{\mathbb{Q}}
	\newcommand{\bt}{\mathbb{T}}
	\newcommand{\bh}{\mathbb{H}}
	\newcommand{\br}{\mathbb{R}}
	\newcommand{\bl}{\mathbf{L}}
	\newcommand{\wt}{\widetilde}
	\newcommand{\im}{{\rm Im}\,}
	\newcommand{\bc}{\mathbb{C}}
	\newcommand{\bp}{\mathbb{P}}
	\newcommand{\ba}{\mathbb{A}}
	\newcommand{\bn}{\mathbb{N}}
	\newcommand{\spin}{{\rm Spin}\,}
	\newcommand{\ds}{\displaystyle}
	\newcommand{\tor}{{\rm Tor}\,}
	\newcommand{\bff}{{\bf F}}
	\newcommand{\bs}{\mathbb{S}}
	\def\ns{\mathop{\lr}}
	\def\nssup{\mathop{\lr\,sup}}
	\def\nsinf{\mathop{\lr\,inf}}
	\renewcommand{\phi}{\varphi}
	\newcommand{\tT}{{\widetilde{T}}}
	\newcommand{\tG}{{\widetilde{G}}}
	\newcommand{\tB}{{\widetilde{B}}}
	\newcommand{\tC}{{\widetilde{C}}}
	\newcommand{\tW}{{\widetilde{W}}}
	\newcommand{\tphi}{{\widetilde{\Phi}}}

	\title{Equivariant cohomology of juggling varieties in rank one}
	\author[Bidhan Paul]{Bidhan Paul}
	\address{School of Mathematical Sciences, Tel Aviv University, Tel Aviv, 69978,
		Israel}
	\email{bidhanam95@gmail.com; bidhanp@tauex.tau.ac.il}



	\begin{abstract}
		We determine the ring structure of the torus-equivariant cohomology of rank-one juggling varieties with rational coefficients. By realizing these varieties as cyclic quiver Grassmannians, we construct a Knutson--Tao type basis for their equivariant cohomology. Using this basis, we give an explicit description of the ring structure in terms of generators and relations, and compute the corresponding structure constants. Finally, we show that these structure constants are integral.
	\end{abstract}

	\maketitle

	\section{Introduction}
The geometry of algebraic varieties arising from representation theory has long served as a bridge between algebraic, combinatorial, and topological structures. Among the most fundamental examples are \emph{quiver Grassmannians} and \emph{affine flag varieties}, which provide geometric models for phenomena in the representation theory of quivers, Lie algebras, and quantum groups. These constructions extend the classical theory of Grassmannians and flag varieties, which have stood for decades as central objects in algebraic geometry and representation theory \cite{carter, k}.

Given a quiver $Q$, a $Q$-representation $M$, and a dimension vector $\mathbf{d}$, the associated quiver Grassmannian $\mbox{Gr}_{\mathbf{d}}(M)$ parametrizes $\mathbf{d}$-dimensional subrepresentations of $M$. These varieties have been studied extensively over the past two decades and have proved useful in diverse areas such as cluster algebras, degenerations of flag varieties, and algebraic combinatorics \cite{cer11, ifr12, ifr13, cer16, cfffr17, sc14}. In particular, the cyclic quiver Grassmannians $X(k,n,\omega)$ (cf. Definition \ref{cqG}) naturally appear in the theory of local models of Shimura varieties \cite{PappasRapoport2005, prs,HainesNgo2002, Gortz2001} and in the study of totally nonnegative Grassmannians and degenerate flag varieties \cite{Feigin, flp24, flp25, p22}. Following \cite{Knutson}, we refer to the varieties $X(k,n,\omega)$ as the \emph{juggling varieties} of rank $n-1$.

A powerful unifying framework for studying the \emph{equivariant cohomology} of such varieties is provided by the \emph{Goresky--Kottwitz--MacPherson (GKM) theory} \cite{gkm, t05}. When a  complex variety $X$ admits an algebraic torus $T$-action with finitely many fixed points and one-dimensional orbits, its equivariant cohomology ring $H_T^*(X)$ can be described combinatorially via its \emph{moment graph}. The vertices of this graph correspond to torus fixed points, while the edges represent one-dimensional orbits labeled by the corresponding tangent weights. This combinatorial model allows one to reconstruct $H_T^*(X)$ entirely from graph data, thereby translating geometric questions into explicit combinatorial computations.

In the case of type~$A$ cyclic quiver Grassmannians, there exists a natural algebraic torus action that satisfies the GKM conditions \cite{lp23, lp231}, enabling a purely combinatorial description of their equivariant cohomology. For integers $n \in \mathbb{N}$, $k \leq n$, and $\omega \geq 1$, the cyclic quiver Grassmannian $X(k,n,\omega)$ admits a natural embedding into the affine flag variety $\mathscr{AF}_n$. Moreover, the affine flag variety $\mathscr{AF}_n$ can be realized as the union of all such quiver Grassmannians $X(k,n,\omega)$. Consequently, the equivariant cohomology of affine flag varieties can be studied and described through the geometry and combinatorics of cyclic quiver Grassmannians.

In this article, we focus on type~$A$ cyclic quiver Grassmannians for $n=2$ ($k=1$ and arbitrary $\omega$) and provide explicit descriptions of their equivariant cohomology rings with rational coefficients in terms of generators and relations. We construct the associated moment graphs, identify the torus fixed points and one-dimensional orbits, and derive explicit GKM-type relations describing their equivariant cohomology. As a corollary, we obtain a description (cf. Corollary \ref{theo:main_afg}) of the equivariant cohomology ring of the affine flag variety of type~$A_1$ (also see \cite{Ande, yun, park2017gkm}). For general $n$, the case is significantly more complex and requires additional algebraic and combinatorial tools; it will be addressed in a future work.


We hope that this paper will be of interest to readers familiar with quiver Grassmannians as well as those working in GKM theory and equivariant geometry. The paper is organized as follows. Section~\ref{s:preli} reviews the necessary background on quiver Grassmannians, affine flag varieties, and GKM theory, and describes the moment graph of $X(1,2,\omega)$. In Section~\ref{s:kt_classes, formula}, we introduce the Knutson–Tao classes for $X(1,2,\omega)$, which form a $\mathbb{Q}[T]$-basis of its equivariant cohomology, and discuss various relations among them. Section~\ref{s:gen_rel} presents the ring structures of the equivariant cohomology of $X(1,2,\omega)$ and of the affine flag variety of type~$A_1$ via generators and relations, using the GKM description. Finally, Section~\ref{s:int_str_cons} discusses the integrality of structure constants of these rings.

\section{Preliminaries}\label{s:preli}
In this section, we summarize several fundamental definitions and facts about quiver Grassmannians \cite{cer11, cer16, sc14}, affine flag varieties \cite{p22,k},  and GKM varieties \cite{gkm} that will be used throughout the paper. In Subsection \ref{ss:mg_of_cqG}, we describe the moment graph of $X(1,2,\omega)$ with respect to $T:=(\bc^*)^2$-action.
	
A \emph{(finite) quiver}  $Q=(Q_0,~Q_1)$ consists of a finite set of vertices $Q_0$ and a finite set of oriented edges (or arrows) $Q_1$ between the vertices. For an oriented edge $e\in Q_1$, we denote the initial and terminal vertices of $e$ by $i(e)$ and $t(e)$ respectively. A \emph{$Q$-representation} $M$ is given by  $((M^{v})_{v\in Q_0},~(M_e)_{e\in Q_1})$, where the $M^{v}$ are $\bc$-vector spaces and each $M_e$ is a linear map from $M^{i(e)}$ to $M^{t(e)}$. Further, the \emph{dimension vector} of $Q$-representation $M$ is \[\mathbf{dim}~ M:= (\mbox{dim}~ M^v)_{v\in Q_0}\in \bz^{Q_0}_{\geq 0}.\]
	
For two $Q$-representations $M$ and $N$, a \emph{$Q$-morphism} $\psi : M \to N$ is a family of linear maps
$(\psi_v:M^v\to N^v)_{v\in Q_0}$ 
such that for every edge $e : v \to w$ in $Q$ the following diagram commutes:
\[\begin{tikzcd} M^{v} \arrow[r,"\psi_v"] \arrow[d,"M_e"'] & N^v \arrow[d,"N_e"] \\ M^w \arrow[r,"\psi_w"'] & N^w \arrow[ul,phantom,"\circlearrowleft"] \end{tikzcd}\]

The collection of all $Q$-morphisms from $M$ to $N$ is denoted by $\mbox{Hom}_Q(M,N)$.  The $Q$-representations, together with these morphisms, form a category, written $\mbox{rep}_{\mathbb{C}}(Q)$. 

A \emph{subrepresentation} $N\subseteq M$ is given by a tuple of vector subspaces $N^v\subseteq M^v$, such that $M_e(N^v)\subseteq N^w$ for all edges $e :v\to w$ of $Q$. 
 
 \begin{defe}[Quiver Grassmannian]\label{def:qG}
For $\mathbf{d}\in \mathbb{Z}_{\geq 0}^{Q_0}$, the \emph{quiver Grassmannian} $\mbox{Gr}_{\mathbf{d}}(M)$ is the variety of all subrepresentations $N\subseteq M$ with $\mathbf{dim}~ N=\mathbf{d}$. In particular, we have
\begin{equation}
	\mbox{Gr}_{\mathbf{d}}(M):=\Big\{(N^v)_{v\in Q_0}\in\prod_{v\in Q_0}	\mbox{Gr}_{\mathbf{d}_v}(M^v) : M_e(N^v)\subseteq N^w,~\forall e :v\to w \in Q_1 \Big\}.
	\end{equation}
 \end{defe}

\begin{exam}
Let $Q$ be an equioriented quiver with $Q_0=\{1,\ldots,n\}$	and $Q_1=\{i\to i+1:\forall i=1,\ldots, n-1\}$. Let the dimension vector $\mathbf{d} = (1,2,\ldots,n)$. 
Consider the $Q$-representation $M$ defined by 
\begin{enumerate}[(i)]
	\item $M^v=\bc^{n+1}$ for any $v\in Q_0$ and, $M_e=\mbox{id}_{\bc^{n+1}}$, for any $e\in Q_1$.
	Then, 	the quiver Grassmannian $\mathrm{Gr}_{\mathbf{d}}(M)$ is isomorphic to the full flag variety $\mathrm{Flag}(n+1)$ in $\mathbb{C}^{n+1}$.
	\item If we relax the map conditions in the above $M$, then $\mathrm{Gr}_{\mathbf{d}}(M)$ is a linear degeneration of  $\mathrm{Flag}(n+1)$ (see, \cite{cfffr17}).
\end{enumerate} 
\end{exam}

If $\mathbf{d}_v > \mbox{dim} ~M^v$ for some $v\in Q_0$, then $\mbox{Gr}_{\mathbf{d}}(M)=\varnothing$.  Further, the variety structure of $\mbox{Gr}_{\mathbf{d}}(M)$ is induced by the embedding into the classical Grassmannian of $\sum_{v\in Q_0} \mathbf{d}_v$-dimensional subspaces of \[V=\bigoplus_{v\in Q_0} M^v,\] and is independent of the choice of bases of the $M^v$'s.

For  an element $N \in \mbox{Gr}_{\mathbf{d}}(M)$, the isomorphism class $[N]$ (known as \emph{stratum}) in $\mbox{Gr}_{\mathbf{d}}(M)$ is irreducible \cite[Lemma 2.4]{ifr12}.
The automorphism group $\mbox{Aut}_Q(M)\subset \mbox{End}_Q(M)$ acts on $\mbox{Gr}_{\mathbf{d}}(M)$ as 
\[A\cdot (N^v)_{v\in Q_0}:= \Big(A_v(N^v)\Big)_{v\in Q_0} \text{for } A\in \mbox{Aut}_Q(M) \text{ and }  N\in \mbox{Gr}_{\mathbf{d}}(M).\]

\subsection{Cyclic quiver Grassmannians}\label{ss:cqG}
Let $\Delta_n$ be the equioriented cycle with vertex set $\mathbb{Z}_n:=\mathbb{Z}/n\mathbb{Z}$ and edges $e: a\to a+1$ for all $a\in \mathbb{Z}_n$. For $m\geq 2$, consider the $\Delta_n$-representation
\begin{equation}\label{defU_m}
	U_m := ((M^{i})_{i\in \bz_n},~(M_a)_{a\in \bz_n}),
\end{equation}
where $M^i=\mathbb{C}^m$ for every $i\in \mathbb{Z}_n$. Denote by $\mathfrak{B}^i=\{b_j^i \mid j\in[m]\}$ the standard basis of the $i$-th copy of $\mathbb{C}^m$. Each arrow map $M_a$ is defined by
\[M_a(b^a_j)=\begin{cases}
	b^{a+1}_{j+1}, & j\in[m-1]\\
	0,&j=m
\end{cases}.\]

\begin{defe}[Cyclic quiver Grassmannian]\label{cqG}
For fixed numbers $k,\omega\geq 1$ with $k\leq n$, we define the cyclic 
 quiver Grassmannian \[X(k,n,\omega):= \mbox{Gr}_{(k\omega,\ldots,k\omega)}(U_{\omega n}).\]
\end{defe}

\begin{rema}
Note that $X(k,n,\omega)$ is a projective variety of dimension $\omega k(n-k)$ \cite[Lemma 4.9]{p22}. Following \cite{Knutson}, we refer to the varieties $X(k,n,\omega)$ as the \emph{juggling varieties} of rank $n-1$. They appear as concrete realizations of the local models of Shimura varieties for $G=GL_n$ and minuscle coweights \cite[Section 7.1]{prs}.
\end{rema}

\begin{rema}\label{embed-qG-to-Gr}
	From Definition \ref{cqG}, one can have the following natural embedding 
	\begin{equation}
	\phi:	X(k,n,\omega)\hookrightarrow \prod_{i=1}^{n} \mbox{Gr}_{k\omega} (M^i),
	\end{equation}
	where for each $i=1,\ldots,n$, $M^i$ are $n\omega$-dimensional vector spaces with bases $\{b^i_j,~j\in[n\omega]\}$ (see \eqref{defU_m}).
\end{rema}

The one-dimensional torus $\mathbb{C}^*$ acts on the vector spaces of $U_{\omega n}$ with weights specified by
$\mbox{wt}(b_j^i):=j$ for all $i\in\bz_n$ and $j\in [\omega n]$. 
By \cite[Lemma 1.1]{cer11}, this action extends naturally to $X(k,n,\omega)$, and the corresponding $\mathbb{C}^*$-fixed point set in $X(k,n,\omega)$ is finite \cite[Theorem 1]{cer11}. 

Moreover, the above action coincides with  a cocharacter of an $(n+1)$-dimensional algebraic torus $(\bc^*)^{n+1}$, which acts on the vector spaces of $U_{\omega n}$ by 
\begin{equation}\label{torusact}
	(t_0,t_1,\ldots,t_n)\cdot b^i_j:=t_0^{j-1}t_{i-j+1}b^i_j \text{ for } (t_0,t_1,\ldots,t_n)\in (\bc^*)^{n+1}.
\end{equation}
This agrees with the torus action on $\Delta_n$-representations as described in \cite{lp23}, and therefore, by \cite[Lemma 5.12]{lp23}, it extends to $X(k,n,\omega)$.

Consider the $1$-dimensional subtorus
\[
T_1 := \{(1,t,\ldots,t) \mid t \in \mathbb{C}^*\} \subset (\mathbb{C}^*)^{n+1},
\]
which acts  by scalar multiplication (cf.  \eqref{torusact}) on the vector spaces of $U_{\omega n}$, and hence trivially on the quiver Grassmannian $X(k,n,\omega)$. 
We therefore consider the action of $n$-dimensional algebraic torus 
\[
T := (\mathbb{C}^*)^{n+1}/T_1 \cong (\mathbb{C}^*)^n.
\]
\begin{lema}
	The set of $T$-fixed points in $X(k,n,\omega)$ coincides with the $\mathbb{C}^*$-fixed point set in $X(k,n,\omega)$ and is finite.
\end{lema}

\begin{proof}
	The first part follows from \cite[Theorem 5.14]{lp23} while the second follows from \cite[Theorem 1]{cer11}.
\end{proof}

We now give an explicit parametrization of the $T$-fixed points of $X(k,n,\omega)$. For $m\leq n$, let $\binom{[n]}{m}$ denote the set of all $m$-element subsets of $[n]$.
\begin{defe}[Generalized juggling patterns]
	For $k,n,\omega\in\bn$ and $k\leq n$, the collection of $(k,n,\omega)$ -juggling patterns is as follows :
	\[\mathscr{J}ug(k,n,\omega):=\Big\{(J_i)_{i\in\bz_n}\in\prod_{i\in\bz_n}\binom{[n\omega]}{k\omega} ~:~J_i+1\subset J_{i+1} \text{ for all } i\in\bz_n\Big\}\]
\end{defe}

\begin{propo}\label{bij:fixedpttojug}(\cite[Lemma 2.8]{flp24})
	The $T$-fixed points in $X(k,n,\omega)$ are in bijection with $\mathscr{J}ug(k,n,\omega)$.
\end{propo}

\subsubsection{Bia\l ynicki--Birula decomposition}
We recall the $\bc^*$-action on $X:=X(k,n,\omega)$ and let $\{v_1,\ldots,v_m\}$ be the $\bc^*$-fixed point set in $X$. This  action gives rise to a stratification 
\begin{equation}
X=\bigcup_{i\in [m]}X_i~~\text{ where, } X_i:=\Big\{v\in X:\substack{\mbox{lim}\\{z\to 0}}~z\cdot v=v_i\Big\},
\end{equation}
This type of decomposition was first studied by Bia\l ynicki--Birula \cite{bb}, and therefore we refer it as \textbf{Bia\l ynicki--Birula decomposition}.

\begin{guess}{\cite[Theorem 2.11]{flp24}}
For $\omega \geq 1$ and $k \leq n$, the variety $X(k,n,\omega)$ satisfies the following properties:
\begin{enumerate}[(i)]
	\item the Bia\l ynicki–Birula (BB) decomposition is a cellular decomposition;
	\item each cell $C\subset X(k,n,\omega)$ is $T$-stable and contains exactly one $T$-fixed point $p_C$;
	\item its irreducible components $X_I(k,n,\omega)$ are indexed by the $k$-element subsets $I \subset [n]$ and are equidimensional;
	\item each $X_I(k,n,\omega)$ is normal, Cohen–Macaulay, and has rational singularities. 
\end{enumerate}
\end{guess}

\subsection{Affine flag varieties}\label{ss:afg}

Let $\widehat{\mathfrak{sl_n}}:=\mathfrak{sl}_n(\mathbb{C}) \otimes_{\mathbb{C}} \mathbb{C}[t, t^{-1}]
\;\oplus\; \mathbb{C}c \;\oplus\; \mathbb{C}d$ be the \emph{affine Kac--Moody Lie algebra} of type $A_{n-1}^{(1)}$,  where $c$ is a central element and $d$ is the degree derivation.  	The Lie bracket is given by
\[\begin{split}
	[x \otimes t^m,\, y \otimes t^n] 
&= [x,y] \otimes t^{m+n} 
+ m\, \delta_{m,-n} \, \mathrm{Tr}(xy)\, c,\\
[c,\, \widehat{\mathfrak{sl}}_n] = 0,&
\qquad 
[d,\, x \otimes t^m] = m\, x \otimes t^m,\end{split}
\]
for all $x, y \in \mathfrak{sl}_n(\mathbb{C})$ and $m,n \in \mathbb{Z}$.

Fix a Cartan decomposition 
\[
\mathfrak{sl}_n = \mathfrak{h} \oplus \mathfrak{n} \oplus \mathfrak{n}_-,
\]
and let $\mathfrak{b} = \mathfrak{h} \oplus \mathfrak{n}$ be the corresponding Borel subalgebra. 
The \emph{Iwahori subalgebra} of $\widehat{\mathfrak{sl}}_n$ is then given by
\[
\mathfrak{b} \otimes 1 \;\oplus\; \mathfrak{sl}_n \otimes t\,\mathbb{C}[t].
\]

Let $\widehat{SL}_n$ denote the affine Kac--Moody group with Lie algebra $\widehat{\mathfrak{sl}}_n$  {\cite{k}}. 
This group contains the finite-dimensional torus $\exp(\mathfrak{h})$ 
and the two-dimensional torus $(\mathbb{C}^*)^2 = \exp(\mathbb{C}c \oplus \mathbb{C}d)$. 
We denote by $\mathfrak{B} \subset \widehat{SL}_n$ the corresponding \emph{Iwahori subgroup}, 
which consists of matrices 
$M(t) \in SL_n(\mathbb{C}[t]) \subset \widehat{SL}_n$ 
such that $M(0)$ is upper triangular. 
In particular, the Lie algebra of $\mathfrak{B}$ is 
$\mathfrak{b} \otimes 1 \;\oplus\; \mathfrak{sl}_n \otimes t\,\mathbb{C}[t]$.

\begin{defe}[affine flag variety]
	The \emph{affine flag variety} for the group  $\widehat{SL}_n$ is defined as the homogeneous space
	\[
	\mathscr{AF}_n:= \widehat{SL}_n / \mathfrak{B}.
	\]
\end{defe}
Let $P_i\subset \widehat{SL}_n, i=1,\ldots,n$ be the maximal parabolic subgroups. Then the \emph{affine Grassmannians} are defined as the quotients $\widehat{SL}_n/P_i$.  One has the natural embedding of the affine flag variety into the product of affine Grassmannians $\mathscr{AF}_n\hookrightarrow \prod_{i=0}^{n-1} \widehat{SL}_n/P_i$.

\begin{rema}Let $W_n$ denote the Weyl group of $\widehat{\mathfrak{sl}}_n$. For $n=2$, the  Weyl group is generated by simple reflections $s_0$ and $s_1$ subject to $s_0^2=s_1^2=e$. For $n>2$, the group is generated by $s_i$ for $i=0,\ldots,n-1$ subject to the relations
	\begin{align*}
		s^2_i=e,~~i=0,\ldots,n-1, ~~~s_is_j=s_js_i,~|i-j|>1,\\
		s_is_{i+1}s_i=s_{i+1}s_is_{i+1}, ~~i=0,\ldots,n-1
	\end{align*}
	considering $s_n=s_0$. The group $W_n$ can also be realized as the group of bijections $f:\bz\to \bz$ subject to the conditions
	\[f(i+n)=f(i)+n \text{ for all $i$ and }\sum_{i=1}^{n}(f(i)-i)=0.\]
	The torus contained in $\widehat{SL}_n$ acts naturally on $\mathscr{AF}_n$, and the corresponding fixed points are parametrized by the elements of $W_n$, the Weyl group of $\widehat{\mathfrak{sl}}_n$.
\end{rema}

\begin{rema}
	For $w\in W_n$, let $p_w$ be the corresponding torus fixed point and  $X_w:= \overline{\mathfrak{B}\cdot p_w}$ be  the corresponding \emph{affine Schubert variety}. Then  $\mathscr{AF}_n=\bigcup_{w\in W_n}X_w$. Therefore, the affine flag variety is an ind-variety i.e. the inductive limit of finite dimensional projective (algebraic) varieties.
\end{rema}

\subsection{Affine flag variery as the inductive limit of quiver Grassmannians}\label{ss:qGinafg}

Let us recall the embedding of the quiver Grassmannian \( X(k,n,\omega) \) introduced in Remark~\ref{embed-qG-to-Gr}.  
By composing with the Pl\"ucker embeddings of the Grassmannians \( \mathrm{Gr}_{k\omega}(M^i) \), we obtain the map  
\[
X(k,n,\omega)\hookrightarrow \prod_{i=1}^{n}\mathbb{P}\big(\Lambda^{k\omega}(M^i)\big).
\]
Our next goal is to construct an embedding of the quiver Grassmannian into a product of Sato Grassmannians.

\begin{defe}[Sato Grassmannians]\label{defSGr}
	For $i\in\bz$, the Sato Grassmannian $\mbox{SGr}_i$ consists of subspaces $V\subset \bc[t,t^{-1}]$ such that 
	\begin{enumerate}[(a)]
		\item $t^N\bc[t^{-1}]\supset V\supset t^{-N}\bc[t^{-1}]$ for some $N\in\bz_{>0}$,
		\item $\mbox{dim}~V/ t^{-N}\bc[t^{-1}]=i+N$.
	\end{enumerate}\end{defe}
\begin{rema}
Sato Grassmannians can be realized as the inductive limit of finite dimensional Grassmann varieties.
\end{rema}	

For each \( i = 1, \ldots, n \), define a linear map 
\[
\psi^i: M^i \longrightarrow \mathbb{C}[t,t^{-1}]
\]
by
\begin{equation}
	\psi^i(b_{n\omega+1-j}^i) = t^{\,i - k\omega + j}, \qquad j \in [n\omega].
\end{equation}
The image of \( \psi^i \) is therefore spanned by the set \( \{t^{\,i - k\omega + j}\}_{j=1}^{n\omega} \).  
We then define an embedding
\begin{equation}
	\Psi^i: \mathrm{Gr}_{k\omega}(M^i) \hookrightarrow \mathbb{C}[t,t^{-1}]
\end{equation}
by
\begin{equation}
	\Psi^i(V^i) = \psi^i(V^i) \oplus \mathrm{span}\{t^j : j \leq i - k\omega\}.
\end{equation}
It follows from~\cite[Lemma~6.2]{flp24} that 
\[
\Psi^i\big(\mathrm{Gr}_{k\omega}(M^i)\big) \subset \mathrm{SGr}_i.
\]
Using the maps \( \Psi^i \) together with the embedding~\eqref{embed-qG-to-Gr}, we obtain the natural morphism
\begin{equation}
	\Psi: X(k,n,\omega)\hookrightarrow \prod_{i=1}^n \mathrm{SGr}_i,
\end{equation}
where \( \Psi = \big(\prod_{i=1}^n \psi^i\big) \circ \phi \).

We can further extend this embedding to an infinite product by adopting the periodicity conventions 
\( M^{i+n} = M^i \) and \( \mathrm{Gr}_{k\omega}(M^i) = \mathrm{Gr}_{k\omega}(M^{i+n}) \), yielding
\begin{equation}\label{embedpsi}
	\Psi: X(k,n,\omega) \hookrightarrow \prod_{i\in\mathbb{Z}} \mathrm{SGr}_i.
\end{equation}

Recall that the irreducible components 
$X_I(k,n,\omega) \subset X(k,n,\omega)$ 
are indexed by the $k$-element subsets $I \subset [n]$.  
Each irreducible component is the closure of a cell $C_I$ 
containing a unique torus fixed point $p_I$.  
The (generalized) juggling pattern $(J_1, \ldots, J_n)$ 
corresponding to the point $p_I$ is given by
\begin{align*}
J_{i_0}
= \big\{\,1 + (i_0 - i) + n(r - 1) \mid i \in I,~ i \le i_0,~ r \in [\omega]\,\big\}\\
\cup 
\big\{\,1 + (i_0 - i) + nr \mid i \in I,~ i > i_0,~ r \in [\omega]\,\big\}
\end{align*}
for all $i_0 \in [n]$. Furthermore, for $I \in \binom{[n]}{k}$, 
we define the following affine Weyl group element $w(I) \in W_n$ so that $\Psi(p_I) = p_{w(I)}$ (cf. \cite[Corollary 6.8]{flp24})
\[
w(I) : i \mapsto 
\begin{cases}
	i - k\omega, & i \notin I,\\[4pt]
	i - k\omega + n\omega, & i \in I,
\end{cases}
\quad \text{for all } i = 1, \ldots, n,
\]

Finally, we recall the following result from~\cite[Theorem 6.14, Proposition 6.15]{flp24}:
\begin{guess}[\cite{flp24}]\label{theo:afv_in_terms_of_qG}
	For each \( k,n, \omega\in \mathbb{N} \) with \( k\leq n \), the quiver Grassmannian \( X(k,n,\omega) \) is isomorphic to the union of Schubert varieties
	\[
	X(k,n,\omega) \;\cong\; \bigcup_{I \in \binom{[n]}{k}} X_{\omega(I)} \subset \mathscr{AF}_n.
	\]
	Moreover, for all integers \( k \) and \( n \), the embedding~\eqref{embedpsi} satisfies the inclusions
	\[
	\Psi\big(X(k,n,\omega)\big) \subset \Psi\big(X(k,n,\omega+1)\big).
	\]
	Finally, the ascending union stabilizes to the affine flag variety:
	\begin{equation}\label{eq:afv_in_terms_of_qG}
	\bigcup_{\omega \ge 1} \Psi\big(X(k,n,\omega)\big) \;=\; \mathscr{AF}_n.
	\end{equation}
\end{guess}

\begin{rema}\label{qg_with_af_pts}
	For $n=2$, the Weyl group elements are of the form $s_0s_1s_0\cdots$ and $s_1s_0s_1\cdots$.  In particular, for any $\omega>0$, there exists exactly two elements $\sigma_1(\omega),~ \sigma_2(\omega)$ of length $\omega$. Further, we have the following 
	\[\Psi(X(1,2,\omega))=X_{\sigma_1(\omega)}\cup X_{\sigma_2(\omega)}\subset \mathscr{AF}_2.\]
\end{rema}

\subsection{GKM varieties and their equivariant cohomology}\label{ss:gkm}
Let $X$ be a projective algebraic variety equipped with an action of an algebraic torus 
\( T := (\mathbb{C}^*)^n \).  
Assume that this action has the property that the set of its $0$- and $1$-dimensional orbits 
has the structure of a graph, and that the odd-dimensional rational cohomology of $X$ vanishes.  
A pair $(X,T)$ satisfying these conditions is called a \emph{GKM variety}.  
This class of varieties was introduced by Goresky, Kottwitz, and MacPherson in~\cite{gkm},  and the term \emph{GKM} derives from their initials. 

Examples of GKM varieties are toric varieties, finite dimensional Schubert varieties of flag varieties of Kac--Moody group and rationally smooth embeddings of reductive groups.

\begin{defe}
Let $(X,T)$ be a GKM variety and let $\chi\in X_*(T)$ be a generic cocharacter (i.e. $X^T=X^{\chi(\bc^*)}$). One can associate a \emph{moment graph} $\mathcal{G}=\mathcal{G}(X,T,\chi)$,
\begin{enumerate}[(i)]
	\item  with vertices $\cv(\cg)=X^T$,
	\item for $x,y\in X^T$, there is an oriented edge $e=(x\to y)\in\ce(\cg)$ if and only if there exists an 1-dimensional $T$-orbit $\mathcal{O}$ such that $\overline{\mathcal{O}}=\mathcal{O}\cup \{x,y\}$ and $\mbox{lim}_{ \lambda\to 0}\chi(\lambda)\cdot q=x$ for $q\in \mathcal{O}$.
	\item each edge $e=(x\to y)\in\ce(\cg)$ is labelled by a character $\alpha(e)\in X^*(T)$ which describes the $T$-action on $\mathcal{O}$.
\end{enumerate}  
\end{defe}

One of the main goals in considering moment graphs of a GKM variety is to describe its equivariant cohomology ring with rational coefficients, 
\( H_T^*(X) \), through the combinatorial properties of the moment graph \cite[Theorem 1.2.2]{gkm}.  These techniques have been found effective to the study of equivariant cohomology of Schubert varieties in Kac--Moody flag varieties, Hessenberg varieties, standard group embeddings and more. 

Let the equivariant cohomology ring of a point be identified with the polynomial ring generated by degree $2$ elements $\tau_1, \ldots, \tau_n$:
\begin{equation}
	\bq[T]:=	H_T^*(\{\mathrm{pt}\}) \simeq \mathbb{Q}[\tau_1, \ldots, \tau_n].
\end{equation}
Here, each generator $\tau_j$ for $j = 1, \ldots, n$ may be regarded as the $\mathbb{Z}$-basis element  of the character lattice $X^*(T)$.  
By abuse of notation, we also denote by $\tau$ its image $\tau \otimes 1$ in the $\mathbb{Q}$-vector space $X^*(T) \otimes_{\mathbb{Z}} \mathbb{Q}$.

\begin{guess}[\cite{gkm}]
	Let $(X,T)$ be a GKM variety with moment graph $\cg=\cg(X,T,\chi)$. Then 
\begin{equation}\label{def:eqcohom}
	 H_T^*(X)\cong\Bigg\{f:\cv(\cg)\to \bq[T]~\Bigg|~ f(x)-f(y)\equiv 0~ \mbox{mod} ~\alpha(e), \forall e:x\to y\in\ce(\cg)\Bigg\}.             
\end{equation}
\end{guess}
The equation $f(x)-f(y)\equiv 0~ \mbox{mod} ~\alpha(e)$ in \eqref{def:eqcohom} is often referred to as a congruence relation. Note that $H^*_T(X)$ has a graded $\bq[T]$-algebra structure induced by 
the injective homomorphism $	\vartheta: \bq[T] \longrightarrow H^*_T(X)$ such that the image of $\tau\in \bq[T] $ (i.e., $\vartheta(\tau): \cv(\cg) \to \bq[T]$) is defined by the function
\begin{equation}\label{def:qt_elements_as_map}
\vartheta(\tau)(x) = \tau \quad \text{for all } x\in \cv(\cg).
\end{equation}

\begin{defe}[Knutson--Tao class]\label{ktclass}
	Let $(X,T)$ be a GKM variety with moment graph $\mathcal{G} := \mathcal{G}(X,T,\chi)$.  
	A \emph{Knutson--Tao class} for a fixed point $x \in \mathcal{V(G)} = X^T$ is an equivariant class 
	\( p_x = (p_x^y)_{y \in X^T} \)
	satisfying the following conditions:
	\begin{enumerate}[(i)]
		\item \( p_x^x = \displaystyle\prod_{e \in \mathcal{E}_x} \alpha(e) \), where $\ce_x\subset \ce$ is the collection of edges with initial vertex $x$. 
		\item for every  \( y \in \mathcal{V(G)} \), the component \( p_x^y \) is a homogeneous polynomial of the same degree in \( \mathbb{Q}[T] \);
	 \item $p_x^y = 0 $ if $x$ can not be reached from $y$ via an oriented path on the graph.
	\end{enumerate}
\end{defe}
The classes in the above definition are named after Knutson and Tao  \cite{kt}, who first introduced classes of this form to construct a basis for the equivariant cohomology of Grassmannians.
\begin{rema}
It is not true in general that Knutson-Tao classes exist.  However, their existence is proved for several smooth GKM varieties in \cite{gz03}, Schubert varieties in \cite[Section 3]{t081}, and BB-filterable GKM varieties in \cite[Theorem 3.9]{lp231}.
\end{rema}

\begin{propo}\label{ktclassformsbasis}(\cite[Proposition 2.13]{t081})
	Let $(X,T)$ be a GKM variety, and let $\chi$ be a generic cocharacter such that the associated moment graph 
	$\mathcal{G}(X,T,\chi)$ is acyclic.  
	If for every fixed point $x \in X^T$ there exists a Knutson--Tao class 
	\( p_x \in H_T^*(X) \),  
	then the collection \( \{\,p_x \mid x \in X^T\,\} \) forms a 
	\( \mathbb{Q}[T] \)-basis of \( H_T^*(X) \).
\end{propo}

\begin{rema}
	For (genaralised) flag varieties, the Knutson--Tao classes are equivariant Schubert classes and the partial order $\succeq$ is the Bruhat order \cite[Proposition 4.6 and Proposition 4.7]{t081}.
\end{rema}

\subsection{Moment graph of cyclic quiver Grassmannian }\label{ss:mg_of_cqG} 

We use the following parametrization of the $T$-fixed points to describe the structure of the one dimensional $T$-orbits.
\begin{defe}
	For $k,n,\omega\in\bn$, with $k\leq n$, we define 
	\begin{equation}
		\mathcal{V}_{k,n,\omega}:=\Big\{(l_j)_{j\in\bz_n}\in [0,\omega n]^{\bz_n}: \mathbf{dim}~\bigoplus_{j\in\bz_n}U_j(l_j)=(k\omega,\ldots,k\omega)\in \bn^{\bz_n}\Big\}.
	\end{equation}
\end{defe}
\begin{propo}\label{bij:vtojug}(\cite[Proposition 3.2]{flp24})
	For $k,n,\omega\in\bn$, with $k\leq n$, there is a bijection between $\mathscr{J}ug(k,n,\omega)$ and $	\mathcal{V}_{k,n,\omega}$.
\end{propo}

For every element $l_\bullet \in \mathcal{V}_{k,n,\omega}$, and for any 
$r \in [0, \min\{l_i,\, \omega n - l_j\}]$ satisfying  $i - l_i \equiv j - l_j - r \pmod{n},$
we define maps 
\[
f_{i,j,r} : \mathcal{V}_{k,n,\omega} \to \mathcal{V}_{k,n,\omega}
\]
by
\[
\big(f_{i,j,r}(l_\bullet)\big)_s :=
\begin{cases}
	l_s, & s \notin \{i,j\},\\[4pt]
	l_i - r, & s = i,\\[4pt]
	l_j + r, & s = j.
\end{cases}
\]
It is straightforward to verify that $f_{i,j,r}(l_\bullet)$ again belongs to $\mathcal{V}_{k,n,\omega}$.
The vertices of the moment graph for the torus action of $T$ on $X(k,n,\omega)$ 
are labelled by the elements of $\mathcal{V}_{k,n,\omega}$ 
(see Propositions~\ref{bij:fixedpttojug} and~\ref{bij:vtojug}).
Moreover, there is an oriented edge in the moment graph from 
$l_\bullet$ to $f_{i,j,r}(l_\bullet)$ if and only if $l_i > l_j + r$ \cite{flp24}.

For $p, p'\in X(k,n,\omega)^T$ we write $p'\preceq 
p$ if $\overline{C_p}$ contains $p'$.   Given two elements $l_\bullet, l_\bullet'\in \cv_{k,n,\omega}$ we write $l_\bullet \geq_\cv l_\bullet'$ if there exists an oriented path from $l_\bullet $ to $l_\bullet'$ in the moment graph for the $T$-action on $X(k,n,\omega)$.  For $J_\bullet, J_\bullet'\in \mathscr{J}ug(k,n,\omega)$ we write $J_\bullet\geq J_\bullet'$ if and only if $j_r^{(i)}\geq j_r'^{(i)}$ for all $i\in \bz_n$ and $r\in [k\omega]$ where we order each $J_i\in \binom{[n\omega]}{k\omega}$ as  \[\big(j_1^{(i)}<j_2^{(i)}<\cdots<j_{k\omega}^{(i)}\big).\]

\begin{guess}(\cite[Theorem 4.6]{flp24})
	For $k,n,\omega \in\bn$ with $k\leq n$, there are order preserving poset isomorphisms between $\cv_{k,n,\omega}$, $\mathscr{J}ug(k,n,\omega)$ and $X(k,n,\omega)^T$.
\end{guess}

For $k=1,n=2,\omega=m$, one has a simpler description of the moment graph  $\mathcal{G}_m$ associated to the $T\cong (\bc^*)^2$-action on $X(1,2,m)$:
\begin{description}
	\item[Vertices] The vertices can be identified with $\cv :=\{(m-q,m+q): -m\leq q\leq m\}$.
	\item[Edges] 
	The set of oriented edges consists of 
	\begin{equation}\label{edge:quiver grass}
		\ce:= \Big\{(m-p,m+p)\to (m-l,m+l) :|p|>|l|, ~|p|-|l| \text{ is odd}\Big\}.
		\end{equation}
	 
	 \item[Label on edges] Further, for $-m \leq q \leq m$, we label the vertex $(m - q, m + q)$ with the weight $(2q + 1)\frac{\alpha}{2}+ \frac{q(q + 1)}{2}\delta$  (cf. Remark \ref{label_on_vertices}). The edges are labelled by the difference between the labels of their source and target vertices. 
	 
	   In particular, for  $-m \leq p,l \leq m$ , the label on the edge $(m-p,m+p)\to (m-l,m+l)\in\ce$ \eqref{edge:quiver grass}  is 
	 \begin{equation}\label{label:edge}
	 	(p-l)\alpha+\frac{(p-l)(p+l+1)}{2}\delta
	 \end{equation}
\end{description}

\begin{rema}\label{label_on_vertices}
Consider the affine Kac--Moody Lie algebra $\widehat{\mathfrak{sl}}_{2}$ with Cartan subalgebra $\widehat{\mathfrak{h}}$.  
Let $\{h_0, h_1, d\}$ be a basis of $\widehat{\mathfrak{h}}$, and let $\{\alpha_0, \alpha_1, \gamma\}$ be the corresponding dual basis of $\widehat{\mathfrak{h}}^{*}$.  
For any $\lambda \in \widehat{\mathfrak{h}}^{*}$, the simple reflections $s_i$ $(i=0,1)$ act as
\[
s_i(\lambda) = \lambda - \lambda(h_i)\,\alpha_i.
\]

Let $\omega_0, \omega_1$ denote the fundamental weights, and set $\Lambda = \omega_0 + \omega_1$ to be the highest weight.  
We fix $\alpha := \alpha_0$ and $\delta := \alpha_0 + \alpha_1$, and identify 
\begin{equation}\label{toruspresent}
	\bq[T]\cong\bq[\alpha,\delta].
	\end{equation}  

For each integer $q = 0, 1, \ldots, m$, we associate the vertices $(m - q, m + q)$ and $(m + q, m - q)$ with Weyl group elements of length $q$, defined respectively by (cf. Remark \ref{qg_with_af_pts})
\[
\sigma_1(q) := s_1 s_0 s_1 \cdots, 
\qquad 
\sigma_2(q) := s_0 s_1 s_0 \cdots.
\]

We then assign to the vertex $(m + q, m -q)$ the weight (cf. \cite[Chapter 17]{carter})
\[
(\sigma_1(q)(\Lambda))(h_1)\, \tfrac{\alpha}{2}
+ (\sigma_1(q)(\Lambda))(-d)\, \delta,
\]
and to the vertex $(m - q, m +q)$ the weight
\[
(\sigma_2(q)(\Lambda))(h_1)\, \tfrac{\alpha}{2}
+ (\sigma_2(q)(\Lambda))(-d)\, \delta.
\]
\end{rema}

\begin{exam}
	For $n=2,k=1 $ and $\omega =4 \text{ or } 5$, the quiver Grassmannians 
$X(k,n,\omega)$ has $9$ and $11$ $T$-fixed  points, respectively. The vertices and their labels for the corresponding $T$-actions on them
	 are listed in Table~\ref{lab:KT_X124} and Table~\ref{lab:KT_X125}, while the associated moment graphs are shown in Figure~\ref{qrep_omega4_fig} and Figure~\ref{qrep_omega5_fig}, respectively.
	
		\begin{table}[H]
		\centering
		\scalebox{1}{
			\begin{tabular}{|c|c|c|c|c|c|c|c|c|}
				\hline
				(8,0) &(7,1)&(6,2)&(5,3)&(4,4)&(3,5)&(2,6)&(1,7)&(0,8)\\ \hline
				$-\frac{7}{2}\alpha+6\delta$&	$-\frac{5}{2}\alpha+3\delta$ &	$-\frac{3}{2}\alpha+\delta$ &	$-\frac{1}{2}\alpha$ &	$\frac{1}{2}\alpha$&$\frac{3}{2}\alpha+\delta$&$\frac{5}{2}\alpha+3\delta$ &$\frac{7}{2}\alpha+6\delta$&$\frac{9}{2}\alpha+10\delta$\\
				\hline
			\end{tabular}
		}
		\caption{vertices and labels on them for the $T$-action on $X(1,2,4)$.}
		\label{lab:KT_X124}
	\end{table}
	\begin{table}[H]
		\centering
		\scalebox{0.9}{
			\begin{tabular}{|c|c|c|c|c|c|c|c|c|c|c|}
				\hline
				(10,0) &(9,1)&(8,2)&(7,3)&(6,4)&(5,5)&(4,6)&(3,7)&(2,8)&(1,9)&(0,10)\\ \hline
				$-\frac{9}{2}\alpha+10\delta$	&$-\frac{7}{2}\alpha+6\delta$&	$-\frac{5}{2}\alpha+3\delta$ &	$-\frac{3}{2}\alpha+\delta$ &	$-\frac{1}{2}\alpha$ &	$\frac{1}{2}\alpha$&$\frac{3}{2}\alpha+\delta$&$\frac{5}{2}\alpha+3\delta$ &$\frac{7}{2}\alpha+6\delta$&$\frac{9}{2}\alpha+10\delta$ &$\frac{11}{2}\alpha+15\delta$\\
				\hline
			\end{tabular}
		}
		\caption{vertices and labels on them for the $T$-action on $X(1,2,5)$.}
		\label{lab:KT_X125}
	\end{table}

		\begin{minipage}{0.4\textwidth}
		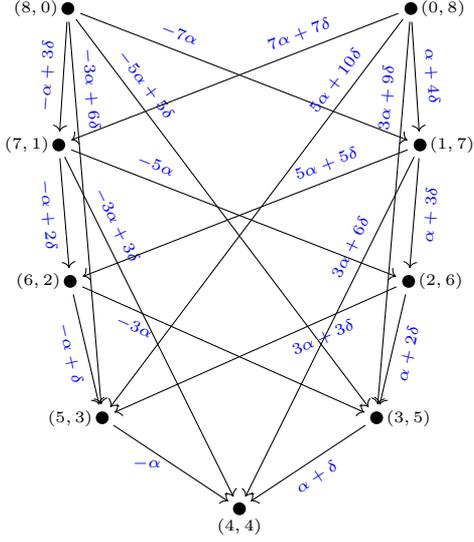
\begin{figure}[H]
			\begin{tikzpicture}
				\begin{scope}[xscale=0.6, yscale=0.6]
					\coordinate (1) at (0,0);
					\coordinate (2) at (-3,2);
					\coordinate (3) at (3,2);
					\coordinate (4) at (-3.7,5);
					\coordinate (5) at (3.7,5);
					\coordinate (6) at (-3.95,8);
					\coordinate (7) at (3.95,8);
					\coordinate (8) at (-3.75,11);
					\coordinate (9) at (3.75,11);

					\fill(1) circle (4pt);
					\node[below] at (1) {\tiny $(4,4)$};
					\fill(2) circle (4pt);
					\node[left] at (2) {\tiny$(5,3)$};
					\fill(3) circle (4pt);
					\node[right] at (3) {\tiny$(3,5)$}; 
					\fill(4) circle (4pt);
					\node[left] at (4) {\tiny$(6,2)$}; 
					\fill(5) circle (4pt);
					\node[right] at (5) {\tiny$(2,6)$}; 
					\fill(6) circle (4pt);
					\node[left] at (6) {\tiny$(7,1)$}; 
					\fill(7) circle (4pt);
					\node[right] at (7) {\tiny$(1,7)$}; 
					\fill(8) circle (4pt);			
					\node[left] at (8) {\tiny$(8,0)$}; 
					\fill(9) circle (4pt);
					\node[right] at (9) {\tiny$(0,8)$}; 
					
					\draw[->,shorten <= 5pt, shorten >= 5pt] (2)--(1) node[midway,left] {\tiny\textcolor{blue}{$-\alpha$}};
					
					\draw[->,shorten <= 5pt, shorten >= 5pt] (3)--(1) node[midway, sloped, below] {\tiny\textcolor{blue}{$\alpha+\delta$}};

					\draw[->,shorten <= 5pt, shorten >= 5pt] (4)--(2) node[midway,sloped, below] {\tiny\textcolor{blue}{$-\alpha+\delta$}};
					
					\draw[->,shorten <= 5pt, shorten >= 5pt] (4)--(3)node[pos=0.23,sloped,below] {\tiny\textcolor{blue}{$-3\alpha$}};
					
					\draw[->,shorten <= 5pt, shorten >= 5pt] (5)--(2)node[pos=0.3,sloped, below] {\tiny\textcolor{blue}{$3\alpha+3\delta$}};
					
					\draw[->,shorten <= 5pt, shorten >= 5pt] (5)--(3)node[midway,sloped, below] {\tiny\textcolor{blue}{$\alpha+2\delta$}};
					
					\draw[->,shorten <= 5pt, shorten >= 5pt] (6)--(4)node[midway,sloped,below] {\tiny\textcolor{blue}{$-\alpha+2\delta$}};
					
					\draw[->,shorten <= 5pt, shorten >= 5pt] (6)--(5)node[pos=0.26, sloped, above] {\tiny\textcolor{blue}{$-5\alpha$}};
					
					\draw[->,shorten <= 5pt, shorten >= 5pt] (6)--(1)node[pos=0.24, sloped, above] {\tiny\textcolor{blue}{$-3\alpha+3\delta$}};
					
					\draw[->,shorten <= 5pt, shorten >= 5pt] (7)--(4)node[pos=0.25, sloped, above] {\tiny\textcolor{blue}{$5\alpha+5\delta$}};

					\draw[->, shorten <= 5pt, shorten >= 5pt] (7)--(5)node[midway, sloped, below] {\tiny\textcolor{blue}{$\alpha+3\delta$}};
					
					\draw[->,shorten <= 5pt, shorten >= 5pt] (7)--(1)node[pos=0.3, sloped, above] {\tiny\textcolor{blue}{$3\alpha+6\delta$}};
					
					\draw[->,shorten <= 5pt, shorten >= 5pt] (8)--(7)node[pos=0.3, sloped, above] {\tiny\textcolor{blue}{$-7\alpha$}};
					
					\draw[->,shorten <= 5pt, shorten >= 5pt] (8)--(3)node[pos=0.21, sloped, above] {\tiny\textcolor{blue}{$-5\alpha+5\delta$}};
					
					\draw[->,shorten <= 5pt, shorten >= 5pt] (8)--(2)node[pos=0.2, sloped, above] {\tiny\textcolor{blue}{$-3\alpha+6\delta$}};
					
					\draw[->,shorten <= 5pt, shorten >= 5pt] (8)--(6)node[midway, sloped, above] {\tiny\textcolor{blue}{$-\alpha+3\delta$}};
					
					\draw[->,shorten <= 5pt, shorten >= 5pt] (9)--(6)node[pos=0.3, sloped, above] {\tiny\textcolor{blue}{$7\alpha+7\delta$}};
					
					\draw[->,shorten <= 5pt, shorten >= 5pt] (9)--(7)node[midway, sloped, above] {\tiny\textcolor{blue}{$\alpha+4\delta$}};
					
					\draw[->,shorten <= 5pt, shorten >= 5pt] (9)--(2)node[pos=0.2, sloped, above] {\tiny\textcolor{blue}{$5\alpha+10\delta$}};
					
					\draw[->, shorten <= 5pt, shorten >= 5pt] (9)--(3)node[pos=0.22, sloped, above] {\tiny \textcolor{blue}{$3\alpha+9\delta$}};		
				\end{scope}
			\end{tikzpicture}
			\caption{the moment graph of $X(1,2,4)$ under the action of $T\cong(\mathbb{C}^*)^2$.}
			\label{qrep_omega4_fig}
		\end{figure}
	\end{minipage}
	\hspace{0.06\textwidth}
	\begin{minipage}{0.5\textwidth}
		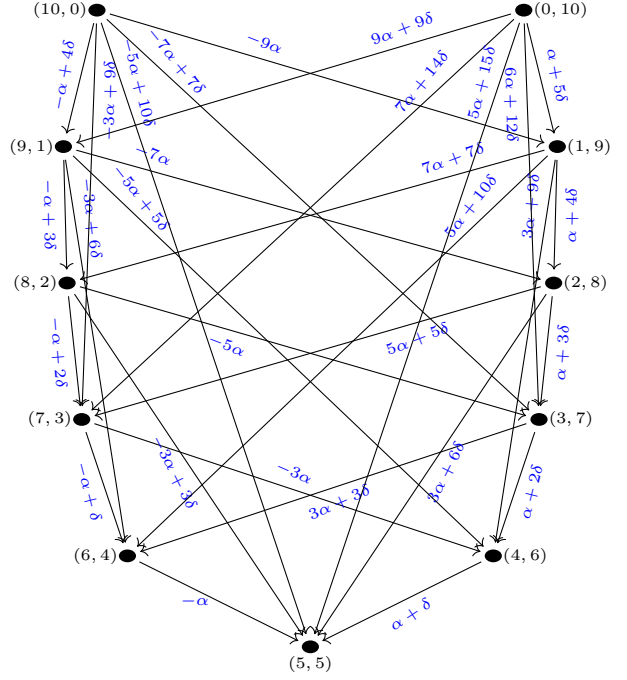
\begin{figure}[H]
			\begin{tikzpicture}
				\begin{scope}[xscale=0.8, yscale=0.6]
					
					\coordinate (1) at (0,0);
					\coordinate (2) at (-3,2);
					\coordinate (3) at (3,2);
					\coordinate (4) at (-3.75,5);
					\coordinate (5) at (3.75,5);
					\coordinate (6) at (-3.99,8);
					\coordinate (7) at (3.99,8);
					\coordinate (8) at (-4.05,11);
					\coordinate (9) at (4.05,11);
					\coordinate (10) at (-3.50,14);
					\coordinate (11) at (3.5,14);

					\fill(1) circle (4pt);
					\node[below] at (1) {\tiny $(5,5)$};
					\fill(2) circle (4pt);
					\node[left] at (2) {\tiny $(6,4)$};
					\fill(3) circle (4pt);
					\node[right] at (3) {\tiny $(4,6)$}; 
					\fill(4) circle (4pt);
					\node[left] at (4) {\tiny $(7,3)$}; 
					\fill(5) circle (4pt);
					\node[right] at (5) {\tiny $(3,7)$}; 
					\fill(6) circle (4pt);
					\node[left] at (6) {\tiny $(8,2)$}; 
					\fill(7) circle (4pt);
					\node[right] at (7) {\tiny $(2,8)$}; 
					\fill(8) circle (4pt);			
					\node[left] at (8) {\tiny $(9,1)$}; 
					\fill(9) circle (4pt);
					\node[right] at (9) {\tiny $(1,9)$}; 
					\node[left] at (10) {\tiny $(10,0)$}; 
					\fill(10) circle (4pt);
					\node[right] at (11) {\tiny $(0,10)$}; 
					\fill(11) circle (4pt);

					\draw[->,shorten <= 5pt, shorten >= 5pt] (2)--(1) node[midway,left] {\tiny\textcolor{blue}{$-\alpha$}};
					
					\draw[->,shorten <= 5pt, shorten >= 5pt] (3)--(1) node[midway, sloped, below] {\tiny\textcolor{blue}{$\alpha+\delta$}};

					\draw[->,shorten <= 5pt, shorten >= 5pt] (4)--(2) node[midway,sloped, below] {\tiny\textcolor{blue}{$-\alpha+\delta$}};
					
					\draw[->,shorten <= 5pt, shorten >= 5pt] (4)--(3)node[midway,sloped,above] {\tiny\textcolor{blue}{$-3\alpha$}};
					
					\draw[->,shorten <= 5pt, shorten >= 5pt] (5)--(2)node[midway,sloped, below] {\tiny\textcolor{blue}{$3\alpha+3\delta$}};
					
					\draw[->,shorten <= 5pt, shorten >= 5pt] (5)--(3)node[midway,sloped, below] {\tiny\textcolor{blue}{$\alpha+2\delta$}};
					
					\draw[->,shorten <= 5pt, shorten >= 5pt] (6)--(4)node[midway,sloped,below] {\tiny\textcolor{blue}{$-\alpha+2\delta$}};
					
					\draw[->,shorten <= 5pt, shorten >= 5pt] (6)--(5)node[pos=0.35, sloped, below] {\tiny\textcolor{blue}{$-5\alpha$}};
					
					\draw[->,shorten <= 5pt, shorten >= 5pt] (6)--(1)node[midway, sloped, below] {\tiny\textcolor{blue}{$-3\alpha+3\delta$}};
					
					\draw[->,shorten <= 5pt, shorten >= 5pt] (7)--(4)node[pos=0.3, sloped, below] {\tiny\textcolor{blue}{$5\alpha+5\delta$}};

					\draw[->, shorten <= 5pt, shorten >= 5pt] (7)--(5)node[midway, sloped, below] {\tiny\textcolor{blue}{$\alpha+3\delta$}};
					
					\draw[->,shorten <= 5pt, shorten >= 5pt] (7)--(1)node[midway, sloped, below] {\tiny\textcolor{blue}{$3\alpha+6\delta$}};

					\draw[->,shorten <= 5pt, shorten >= 5pt] (8)--(6)node[midway, sloped, below] {\tiny\textcolor{blue}{$-\alpha+3\delta$}};
					
					\draw[->,shorten <= 5pt, shorten >= 5pt] (8)--(7)node[pos=0.17, sloped, above] {\tiny\textcolor{blue}{$-7\alpha$}};
					
					\draw[->,shorten <= 5pt, shorten >= 5pt] (8)--(2)node[pos=0.18, sloped, above] {\tiny\textcolor{blue}{$-3\alpha+6\delta$}};
					
					\draw[->,shorten <= 5pt, shorten >= 5pt] (8)--(3)node[pos=0.15, sloped, above] {\tiny\textcolor{blue}{$-5\alpha+5\delta$}};
					
					\draw[->,shorten <= 5pt, shorten >= 5pt] (9)--(6)node[pos=0.2, sloped, above] {\tiny\textcolor{blue}{$7\alpha+7\delta$}};
					
					\draw[->,shorten <= 5pt, shorten >= 5pt] (9)--(7)node[midway, sloped, below] {\tiny\textcolor{blue}{$\alpha+4\delta$}};
					
					\draw[->,shorten <= 5pt, shorten >= 5pt] (9)--(2)node[pos=0.17, sloped, above] {\tiny\textcolor{blue}{$5\alpha+10\delta$}};
					
					\draw[->, shorten <= 5pt, shorten >= 5pt] (9)--(3)node[pos=0.15, sloped, above] {\tiny\textcolor{blue}{$3\alpha+9\delta$}};		
					
					\draw[->, shorten <= 5pt, shorten >= 5pt] (10)--(8)node[midway, sloped, above] {\tiny\textcolor{blue}{$-\alpha+4\delta$}};	
					
					\draw[->, shorten <= 5pt, shorten >= 5pt] (10)--(9)node[pos=0.35, sloped, above] {\tiny\textcolor{blue}{$-9\alpha$}};
					
					\draw[->, shorten <= 5pt, shorten >= 5pt] (10)--(4)node[pos=0.22, sloped, below] {\tiny\textcolor{blue}{$-3\alpha+9\delta$}};	
					
					\draw[->, shorten <= 5pt, shorten >= 5pt] (10)--(5)node[pos=0.15, sloped, above] {\tiny\textcolor{blue}{$-7\alpha+7\delta$}};
					
					\draw[->, shorten <= 5pt, shorten >= 5pt] (10)--(1)node[pos=0.12, sloped, above] {\tiny\textcolor{blue}{$-5\alpha+10\delta$}};

					\draw[->, shorten <= 5pt, shorten >= 5pt] (11)--(8)node[pos=0.25, sloped, above] {\tiny\textcolor{blue}{$9\alpha+9\delta$}};	
					
					\draw[->, shorten <= 5pt, shorten >= 5pt] (11)--(9)node[midway, sloped, above] {\tiny\textcolor{blue}{$\alpha+5\delta$}};
					
					\draw[->, shorten <= 5pt, shorten >= 5pt] (11)--(4)node[pos=0.2, sloped, above] {\tiny\textcolor{blue}{$7\alpha+14\delta$}};	
					
					\draw[->, shorten <= 5pt, shorten >= 5pt] (11)--(5)node[pos=0.23, sloped,below] {\tiny\textcolor{blue}{$6\alpha+12\delta$}};
					
					\draw[->, shorten <= 5pt, shorten >= 5pt] (11)--(1)node[pos=0.12, sloped, above] {\tiny\textcolor{blue}{$5\alpha+15\delta$}};

				\end{scope}
			\end{tikzpicture}
			\caption{the moment graph of $X(1,2,5)$ under the action of $T\cong(\mathbb{C}^*)^2$.}
			\label{qrep_omega5_fig}
		\end{figure}
	\end{minipage}
\end{exam}


	\section{KT classes, explicit formulas for the KT classes in the $n=2$ case.}\label{s:kt_classes, formula}
From now on, we focus on the quiver Grassmannians
$X_m := X(1,2,m) \quad \text{(i.e., } n=2, k=1, \omega=m\text{)}$
under the action of the torus $T \cong (\mathbb{C}^*)^2$.
In this section, we provide explicit formulas for the Knutson--Tao classes associated with the moment graph of $X_m$, which form a basis of its torus-equivariant cohomology. We also describe several relations among these classes, which are crucial for computing the equivariant cohomology ring of $X_m$ and, subsequently, of the affine flag variety of type $A_1$.

	\begin{defe}\label{ktclasses}
		For $1\leq q\leq m$, we define the maps
		\begin{align*}
	\xi(m+q,\,m-q) : \mathcal{V} \to \bq[\alpha,\delta], 
\qquad &
\xi(m-q,\,m+q) : \mathcal{V} \to \bq[\alpha,\delta]\\
&\\
\text{and, } \xi(m,m&):\cv \to \bq[\alpha,\delta]
		\end{align*}
		 respectively by
		\begin{equation}\label{eq:xi_plus_minus}
		\xi(m+q,m-q)(x)=
		\begin{cases}
		\displaystyle	\binom{\lceil \frac{p-q+1}{2} \rceil + q-1}{\,q\,}\prod_{k=\lceil \frac{p-q+1}{2} \rceil -1}^{\lceil \frac{p-q+1}{2} \rceil +q-2}\Big(-\alpha +k\delta\Big)                  &      \text{if $x=(m+p,m-p), q\leq p\leq m$}  \\		
			&\\
		\displaystyle	\binom{\lceil \frac{p-q}{2} \rceil+ q-1}{\,q\,}\prod_{k=\lceil \frac{p-q}{2} \rceil+1}^{\lceil \frac{p-q}{2} \rceil+q}\Big(\alpha+k\delta \Big)  &      \text{if $x=(m-p,m+p), q+1\leq p\leq m$}  \\
			&\\
			~~~~~~~0&      \text{otherwise}  
		\end{cases}
	\end{equation}
		\begin{equation}\label{eq:xi_minus_plus}
		\xi(m-q,m+q)(x)=
		\begin{cases}
		\displaystyle	\binom{\lceil \frac{p-q}{2} \rceil + q-1}{\,q\,}\prod_{k=\lceil \frac{p-q}{2} \rceil}^{\lceil \frac{p-q}{2} \rceil +q-1}\Big(-\alpha +k\delta\Big)                   &      \text{if $x=(m+p,m-p), q+1\leq p\leq m$}  \\		
			&\\
			\displaystyle\binom{\lceil \frac{p-q+1}{2} \rceil+ q-1}{\,q\,}\prod_{k=\lceil \frac{p-q+1}{2} \rceil}^{\lceil \frac{p-q+1}{2} \rceil+q-1 }\Big(\alpha+k\delta \Big)&      \text{if $x=(m-p,m+p), q\leq p\leq m$}  \\
			&\\
			~~~~~~~0&      \text{otherwise}  
		\end{cases}
			\end{equation}

				\begin{equation}\label{eq:xi_mm}
			\text{and }~~~~	\xi(m,m)(x)=1 \text{ for any } x\in\cv.
			\end{equation}
	\end{defe}

\begin{guess}\label{gen_in_cohom-ring}
For $1 \leq q \leq m$, the maps $\xi(m+q,m-q)$, $\xi(m-q,m+q)$, and $\xi(m,m)$ are elements of $H_T^*(X_m)$.
\end{guess}
We postpone the proof of Theorem~\ref{gen_in_cohom-ring} until the end of this section, after we establish some more results.

\begin{defe}\label{Def:p^_}
	For any $0 \leq q \leq m$, we define
	\[
	p_{(m+q,m-q)}^{(2m-i,i)} := \xi(m+q,m-q)\big((2m-i,i)\big), \qquad 0 \leq i \leq 2m,
	\]
	and
	\[
	p_{(m-q,m+q)}^{(2m-i,i)} := \xi(m-q,m+q)\big((2m-i,i)\big), \qquad 0 \leq i \leq 2m.
	\]
At several points in the paper, we shall interpret 
$p_{(m+q,m-q)}^{(2m-i,i)},~p_{(m-q,m+q)}^{(2m-i,i)} \in \mathbb{Q}[T] (\cong\bq[\alpha,\delta])$ 
as the constant maps 
$p_{(m+q,m-q)}^{(2m-i,i)} \cdot \xi(m,m)$ 
and 
$p_{(m-q,m+q)}^{(2m-i,i)} \cdot \xi(m,m)$, 
respectively.
\end{defe}

\begin{rema}\label{prod_of_edge_label}
	In particular, for any $0\leq q\leq m$,  $p_{(m+q,m-q)}^{(m+q,m-q)}$ (resp. $p_{(m-q,m+q)}^{(m-q,m+q)}$) is given by the product of the labels (upto a scalar multiple) of all outgoing edges  from the vertex $(m+q,m-q)$ (resp. $(m-q,m+q)$).
\end{rema}

Note that the maps $\xi(m+q,m-q)$ and $\xi(m-q,m+q)$ for $q = 1, \ldots, m$, together with $\xi(m,m)$, satisfy the conditions of Definition~\ref{ktclass}. 
Hence, by Proposition~\ref{ktclassformsbasis}, we obtain that  
\[
\Big\{\xi(m,m),~\xi(m+q,m-q),~\xi(m-q,m+q) : 1 \le q \le m\Big\}
\]
forms a $\bq[\alpha,\delta]$- basis of $H_T^*(X_m)$ (see also Lemma~\ref{Lem:strconsts} and Remark~\ref{ktclass_indep}). 
We refer to this as the \emph{Knutson--Tao type basis}.

	\begin{exam}		
For the quiver Grassmannian $X(1,2,4)$ under the action of $T \cong    (\mathbb{C}^*)^2$, 
the Knutson--Tao type basis elements $\{\xi(a,b)\}$ are listed in Table~\ref{tab:KT_X124}.
	\begin{table}[H]
	\centering
	\scalebox{0.8}{
	\begin{tabular}{|c|c|c|c|c|c|c|c|c|c|}
	\hline
	&(8,0) &(0,8)&(7,1)&(1,7)&(6,2)&(2,6)&(5,3)&(3,5)&(4,4)\\ \hline
	$\xi(8,0)$ &$\displaystyle\prod_{k=0}^{3}(-\alpha+k\delta)$ &0&0&0&0&0&0&0&0\\ \hline
	$\xi(0,8)$&0 & $\displaystyle\prod_{k=1}^{4}(\alpha+k\delta)$&0&0&0&0&0&0&0\\ \hline
	$ \xi(7,1)$&$\displaystyle\prod_{k=0}^{2}(-\alpha+k\delta)$ &$\displaystyle\prod_{k=2}^{4}(\alpha+k\delta)$ &$\displaystyle\prod_{k=0}^{2}(-\alpha+k\delta)$&0&0&0&0&0&0\\\hline
	 $\xi(1,7)$&$\displaystyle\prod_{k=1}^{3}(-\alpha+k\delta)$ &$\displaystyle\prod_{k=1}^{3}(\alpha+k\delta)$&0 &$\displaystyle\prod_{k=1}^{3}(\alpha+k\delta)$&0&0&0&0&0\\ \hline
	 $\xi(6,2)$&$\displaystyle 3\prod_{k=1}^{2}(-\alpha+k\delta)$ &$\displaystyle\prod_{k=2}^{3}(\alpha+k\delta)$ &$\displaystyle\prod_{k=0}^{1}(-\alpha+k\delta)$&$\displaystyle\prod_{k=2}^{3}(\alpha+k\delta)$&$\displaystyle\prod_{k=0}^{1}(-\alpha+k\delta)$&0&0&0&0\\ \hline
	 
	  $\xi(2,6)$&$\displaystyle\prod_{k=1}^{2}(-\alpha+k\delta)$ &$\displaystyle 3 \prod_{k=2}^{3}(\alpha+k\delta)$ &$\displaystyle\prod_{k=1}^{3}(\alpha+k\delta)$&$\displaystyle\prod_{k=1}^{2}(-\alpha+k\delta)$&0&$\displaystyle\prod_{k=1}^{2}(\alpha+k\delta)$&0&0&0\\  \hline 
	  
	  $\xi(5,3)$&$2(-\alpha+\delta)$ &$2(\alpha+3\delta)$&$2(-\alpha+\delta)$&$(\alpha+2\delta)$ &$-\alpha$ &$\alpha+2\delta$ &$-\alpha$ &0&0\\ \hline
	  $\xi(3,5)$&$2(-\alpha+2\delta)$ &$2(\alpha+2\delta)$ &$-\alpha+\delta$ &$2(\alpha+2\delta)$ &$-\alpha+\delta$ &$\alpha+\delta$ &0&$\alpha+\delta$ &0\\ \hline
	  $\xi(4,4)$ &1&1&1&1&1&1&1&1&1\\ \hline
\end{tabular}
}
\caption{Values of the Knutson--Tao type basis elements $\xi(a,b)$ for $X(1,2,4)$ under the action of $T\cong   (\mathbb{C}^*)^2$.}
\label{tab:KT_X124}
\end{table}
\end{exam}

\begin{exam}		
	For the quiver Grassmannian $X(1,2,5)$ under the action of $T =    (\mathbb{C}^*)^2$, 
	the Knutson--Tao type basis elements $\{\xi(a,b)\}$ are listed in Table~\ref{tab:KT_X125}.
	\begin{table}[H]
		\centering
		\scalebox{0.68}{
			\begin{tabular}{|c|c|c|c|c|c|c|c|c|c|c|c|}
				\hline
				&(10,0) &(0,10)&(9,1)&(1,9)&(8,2)&(2,8)&(7,3)&(3,7)&(6,4)&(4,6)&(5,5)\\ \hline
				$\xi(10,0)$ &$\displaystyle\prod_{k=0}^{4}(-\alpha+k\delta)$ &0&0&0&0&0&0&0&0&0&0\\ \hline
				$\xi(0,10)$&0 & $\displaystyle\prod_{k=1}^{5}(\alpha+k\delta)$&0&0&0&0&0&0&0&0&0\\ \hline
				
				$\xi(9,1)$ &$\displaystyle\prod_{k=0}^{4}(-\alpha+k\delta)$ &$\displaystyle\prod_{k=2}^{5}(\alpha+k\delta)$&$\displaystyle\prod_{k=0}^{4}(-\alpha+k\delta)$ &0&0&0&0&0&0&0&0\\ \hline
				
				$\xi(1,9)$& $\displaystyle\prod_{k=1}^{4}(-\alpha+k\delta)$& $\displaystyle\prod_{k=1}^{5}(\alpha+k\delta)$&0 & $\displaystyle\prod_{k=1}^{5}(\alpha+k\delta)$&0&0&0&0&0&0&0\\ \hline
				
				$\xi(8,2)$&$4\displaystyle\prod_{k=1}^{3}(-\alpha+k\delta)$ &$\displaystyle\prod_{k=2}^{4}(\alpha+k\delta)$&$\displaystyle\prod_{k=0}^{2}(-\alpha+k\delta)$ &$\displaystyle\prod_{k=2}^{4}(\alpha+k\delta)$ &$\displaystyle\prod_{k=0}^{2}(-\alpha+k\delta)$&0&0&0&0&0&0\\\hline
				
				$\xi(2,8)$&$\displaystyle\prod_{k=1}^{3}(-\alpha+k\delta)$&$4\displaystyle\prod_{k=2}^{4}(\alpha+k\delta)$&$\displaystyle\prod_{k=1}^{3}(-\alpha+k\delta)$ &$\displaystyle\prod_{k=1}^{3}(\alpha+k\delta)$&0 &$\displaystyle\prod_{k=1}^{3}(\alpha+k\delta)$&0&0&0&0&0\\ \hline
				
				$\xi(7,3)$&$3\displaystyle\prod_{k=1}^{2}(-\alpha+k\delta)$&$3\displaystyle\prod_{k=3}^{4}(\alpha+k\delta)$&$\displaystyle 3\prod_{k=1}^{2}(-\alpha+k\delta)$ &$\displaystyle\prod_{k=2}^{3}(\alpha+k\delta)$ &$\displaystyle\prod_{k=0}^{1}(-\alpha+k\delta)$&$\displaystyle\prod_{k=2}^{3}(\alpha+k\delta)$&$\displaystyle\prod_{k=0}^{1}(-\alpha+k\delta)$&0&0&0&0\\ \hline
				
				$\xi(3,7)$&$3\displaystyle\prod_{k=2}^{3}(-\alpha+k\delta)$&$3\displaystyle\prod_{k=2}^{3}(\alpha+k\delta)$&$\displaystyle\prod_{k=1}^{2}(-\alpha+k\delta)$ &$\displaystyle 3 \prod_{k=2}^{3}(\alpha+k\delta)$ &$\displaystyle\prod_{k=1}^{3}(\alpha+k\delta)$&$\displaystyle\prod_{k=1}^{2}(-\alpha+k\delta)$&0&$\displaystyle\prod_{k=1}^{2}(\alpha+k\delta)$&0&0&0\\  \hline 
				
				$\xi(6,4)$&$3(-\alpha+2\delta)$&$2(\alpha+3\delta)$&$2(-\alpha+\delta)$ &$2(\alpha+3\delta)$&$2(-\alpha+\delta)$&$(\alpha+2\delta)$ &$-\alpha$ &$\alpha+2\delta$ &$-\alpha$ &0&0\\ \hline
				
				$\xi(4,6)$&$2(-\alpha+2\delta)$&$3(\alpha+3\delta)$&$2(-\alpha+2\delta)$ &$2(\alpha+2\delta)$ &$-\alpha+\delta$ &$2(\alpha+2\delta)$ &$-\alpha+\delta$ &$\alpha+\delta$ &0&$\alpha+\delta$ &0\\ \hline
				
				$\xi(5,5)$ &1&1&1&1&1&1&1&1&1&1&1\\ \hline
			\end{tabular}
		}
		\caption{Values of the Knutson--Tao type basis elements $\xi(a,b)$ for $X(1,2,5)$ under the action of $T\cong (\mathbb{C}^*)^2$.}
		\label{tab:KT_X125}
	\end{table}
\end{exam}

		\begin{propo}\label{mult:deg2deg2}
			We have the following realation:
		\begin{equation}\label{mult22}
		\xi(m+1,m-1)\cdot\xi(m-1,m+1)=	\xi(m+2,m-2)+\xi(m-2,m+2).
		\end{equation}
	\end{propo}

	\begin{proof}
		We verify the result by evaluating \eqref{mult22} for each $x\in\cv$.
\begin{enumerate}[(i)]
	\item if $x=(m+p,m-p)$, $p\geq 2$: 
	\[\begin{split}
\text{LHS}=&\lceil \frac{p}{2}\rceil \Big(-\alpha +(\lceil \frac{p}{2}\rceil -1)\delta\Big)
\lceil \frac{p-1}{2} \rceil \Big(-\alpha +\lceil \frac{p-1}{2} \rceil \delta\Big),\\
\text{RHS}=&\binom{\lceil \frac{p-1}{2} \rceil + 1}{\,2\,}\displaystyle\prod_{k=\lceil \frac{p-1}{2} \rceil -1}^{\lceil \frac{p-1}{2} \rceil }\big(-\alpha +k\delta\big)      +	\binom{\lceil \frac{p-2}{2} \rceil + 1}{\,2\,}\displaystyle\prod_{k=\lceil \frac{p-2}{2} \rceil}^{\lceil \frac{p-2}{2} \rceil +1 }\big(-\alpha +k\delta\big).	\end{split}\]
Depending on the parity of $p$, it is straightforward to verify that the left-hand side and right-hand side are equal.

	\item if $x=(m-p,m+p)$, $p\geq 2$: 
	\[\begin{split}
		\text{LHS}=&	\lceil \frac{p-1}{2}\rceil\Big(\alpha+(\lceil \frac{p-1}{2}\rceil+1)\delta \Big) 	{\lceil \frac{p}{2} \rceil}\Big(\alpha+\lceil \frac{p}{2} \rceil\delta \Big) ,\\
		\text{RHS}=&\binom{\lceil \frac{p-2}{2} \rceil+ 1}{\,2\,}\displaystyle\prod_{k=\lceil \frac{p-2}{2} \rceil+1}^{\lceil \frac{p-2}{2} \rceil+2}\Big(\alpha+k\delta \Big)+\binom{\lceil \frac{p-1}{2} \rceil+ 1}{\,2\,}\displaystyle\prod_{k=\lceil \frac{p-1}{2} \rceil}^{\lceil \frac{p-1}{2} \rceil+1 }\Big(\alpha+k\delta \Big) 
	\end{split}\]
	Depending on the parity of $p$, it is straightforward to verify that the left-hand side and right-hand side are equal.
	
\item For all other vertices, both sides of the expression are zero.
	\end{enumerate}	
\end{proof}

\begin{propo}\label{mult:highest}
	For any $0\leq q\leq m$, we have the following multiplicative relations:
\begin{equation}\label{mult:2m0}
	\xi(2m,0)\Big(\xi(m+q,m-q)-p_{(m+q,m-q)}^{(2m,0)}\Big)=0,	\qquad \xi(2m,0)\Big(\xi(m-q,m+q)-p_{(m-q,m+q)}^{(2m,0)}\Big)=0
\end{equation}
\begin{equation}\label{mult:02m}
	\xi(0,2m)\Big(\xi(m+q,m-q)-p_{(m+q,m-q)}^{(0,2m)}\Big)=0,	\qquad \xi(0,2m)\Big(\xi(m-q,m+q)-p_{(m-q,m+q)}^{(0,2m)}\Big)=0.
\end{equation}
\end{propo}
\begin{proof}
	Since $\xi(2m,0)$ vanishes everywhere except at $(2m,0)$, and since (cf.  Definition  \ref{Def:p^_}),  
	\[
	\xi(m+q,m-q)(2m,0) = p_{(m+q,m-q)}^{(2m,0)}, \qquad 
	\xi(m-q,m+q)(2m,0) = p_{(m-q,m+q)}^{(2m,0)},
	\]
	it follows that \eqref{mult:2m0} holds.  
	
A similar argument shows that \eqref{mult:02m} holds.
\end{proof}

\begin{propo} \label{mult:q_odd_propo}
For each odd integer $q$ satisfying $0 < q < m$, we have the following multiplicative relations:
\begin{equation}\label{mult:odd+-+-}
		\xi(m+q,m-q)\Big(\xi(m+1,m-1)- p^{(m+q,m-q)}_{(m+1,m-1)}\Big)=0\cdot\xi(m+q+1,m-q-1)+(q+1)\cdot\xi(m-q-1,m+q+1)
\end{equation}
\begin{equation}\label{mult:odd-+-+}
		\xi(m-q,m+q)\Big(\xi(m-1,m+1)-p^{(m-q,m+q)}_{(m-1,m+1)}\Big)=(q+1)\cdot\xi(m+q+1,m-q-1)+0\cdot\xi(m-q-1,m+q+1)
\end{equation}
\end{propo}
	
	\begin{proof}
	\underline{Step 1:	Equality of  \eqref{mult:odd+-+-}}

Note that the right-hand side of \eqref{mult:odd+-+-}  vanishes when $x = (m+p, m-p)$ for $0 \leq p \leq q+1$, 
and when $x = (m-p, m+p)$ for $0 \leq p \leq q$. 

Furthermore, 	$\xi(m+q,m-q)(x)=0$ for $x = (m+p, m-p), 0 \leq p \leq q-1$, and  for $x = (m-p, m+p),0 \leq p \leq q$. Since $q$ is odd,
\[
\xi(m+1, m-1)\big((m+q, m-q)\big)
= p^{(m+q, m-q)}_{(m+1, m-1)}= \xi(m+1, m-1)\big((m+q+1, m-q-1)\big).
\]
It follows that the left-hand side of \eqref{mult:odd+-+-} also vanishes on each of these vertices. 

Next, we verify \eqref{mult:odd+-+-} by evaluating it on the remaining vertices $x \in \mathcal{V}$.

			\begin{enumerate}[(i)]
				\item if $x=(m+p,m-p), q+2\leq p\leq m$:
				\[\begin{split}
					\text{LHS}=&\binom{\lceil \frac{p-q+1}{2} \rceil + q-1}{\,q\,}\displaystyle\prod_{k=\lceil \frac{p-q+1}{2} \rceil -1}^{\lceil \frac{p-q+1}{2} \rceil +q-2}\Big(-\alpha +k\delta\Big)  \cdot    (\lceil \frac{p}{2}\rceil-\lceil \frac{q}{2}\rceil)    \Big(-\alpha +(\lceil \frac{p}{2}\rceil+\lceil \frac{q}{2}\rceil-1)\delta \Big)\\
					=& \displaystyle \lceil \frac{p-q-1}{2}\rceil \cdot\binom{\lceil \frac{p-q-1}{2} \rceil + q}{\,q\,}\prod_{k=\lceil \frac{p-q-1}{2} \rceil }^{\lceil \frac{p-q-1}{2} \rceil +q-1}\Big(-\alpha +k\delta\Big)  \cdot   \Big(-\alpha +(\lceil \frac{p-q-1}{2}\rceil+q)\delta \Big)  \text{ ~~[cf. Remark \ref{useful_rel_q_odd_ii}]} \\
					=&(q+1)	\binom{\lceil \frac{p-q-1}{2} \rceil + q}{\,q+1\,}\displaystyle\prod_{\lceil \frac{p-q-1}{2} \rceil}^{\lceil \frac{p-q-1}{2} \rceil +q}\Big(-\alpha +k\delta\Big)~~~~~ \text{\Big [since, $(n-r)\binom{n}{r}=(r+1)\binom{n}{r+1}$\Big]}\\
				=&\text{RHS}
				\end{split}\]

				\item if $x=(m-p,m+p), q+1\leq p\leq m$:	
				\[\begin{split}
					\text{LHS}	=& 		\displaystyle	\binom{\lceil \frac{p-q}{2} \rceil + q-1}{\,q\,}\prod_{k=\lceil \frac{p-q}{2} \rceil +1}^{\lceil \frac{p-q}{2} \rceil +q}\Big(\alpha +k\delta\Big)  \cdot  \Big(\lceil \frac{p-1}{2}\rceil+ \lceil \frac{q}{2}\rceil\Big)\Big(\alpha+(\lceil \frac{p-1}{2}\rceil- \lceil \frac{q}{2}\rceil+1)\delta\Big)\\
					=&	\displaystyle\Big( \lceil\frac{p-q}{2}\rceil+q\Big)	\binom{\lceil \frac{p-q}{2} \rceil + q-1}{\,q\,}\prod_{k=\lceil \frac{p-q}{2} \rceil +1}^{\lceil \frac{p-q}{2} \rceil +q}\Big(\alpha +k\delta\Big)  \cdot  \Big(\alpha+\lceil \frac{p-q}{2}\rceil\delta\Big) \text{ ~~~~~~[cf. Remark \ref{useful_rel_q_odd_i}]}  \\
					=&(q+1)	\displaystyle\binom{\lceil \frac{p-q}{2} \rceil+ q}{\,q+1\,}\prod_{k=\lceil \frac{p-q}{2} \rceil}^{\lceil \frac{p-q}{2} \rceil+q }\Big(\alpha+k\delta \Big)  ~~~~~ \text{\Big [since, $(n+1)\binom{n}{r}=(r+1)\binom{n+1}{r+1}$\Big]}  \\
					= &\text{RHS}
				\end{split}\]			
			\end{enumerate}

				\underline{Step 2:	Equality of  \eqref{mult:odd-+-+}}

Note that the right-hand side of \eqref{mult:odd-+-+}  vanishes when $x = (m+p, m-p)$ for $0 \leq p \leq q$, 
and when $x = (m-p, m+p)$ for $0 \leq p \leq q+1$.  

	Furthermore, 	$\xi(m-q,m+q)(x)=0$ for $x = (m+p, m-p), 0 \leq p \leq q$, and  for $x = (m-p, m+p),0 \leq p \leq q-1$. Since $q$ is odd,
\[
\xi(m-1, m+1)\big((m-q, m+q)\big)
= p^{(m-q, m+q)}_{(m-1, m+1)}= \xi(m-1, m+1)\big((m-q-1, m+q+1)\big)
\]
It follows that the left-hand side of \eqref{mult:odd-+-+} also vanishes on each of these vertices.  Next, we verify \eqref{mult:odd-+-+} by evaluating it on the remaining vertices $x \in \mathcal{V}$.

	\begin{enumerate}[(i)]
	\item if $x=(m+p,m-p), q+1\leq p\leq m$:
		\[\begin{split}
		\text{LHS}
		=& \displaystyle	\binom{\lceil \frac{p-q}{2} \rceil + q-1}{\,q\,}\prod_{k=\lceil \frac{p-q}{2} \rceil}^{\lceil \frac{p-q}{2} \rceil +q-1}\Big(-\alpha +k\delta\Big)  \cdot\Big( \lceil \frac{p-1}{2} \rceil +\lceil \frac{q}{2} \rceil\Big) \big(-\alpha +\big(\lceil \frac{p-1}{2} \rceil -\lceil \frac{q}{2} \rceil\big)\delta\big)     \\
		=&  \displaystyle       \big(\lceil\frac{p-q}{2}\rceil+q\big)           \binom{\lceil \frac{p-q}{2} \rceil + q-1}{\,q\,}              \prod_{k=\lceil \frac{p-q}{2} \rceil -1}^{\lceil \frac{p-q}{2} \rceil +q-1}\Big(-\alpha +k\delta\Big)        \text{ ~~~~~~[cf. Remark \ref{useful_rel_q_odd_i}]}  \\
		=&(q+1) \displaystyle	\binom{\lceil \frac{p-q}{2} \rceil + q}{\,q+1\,}\prod_{k=\lceil \frac{p-q}{2} \rceil -1}^{\lceil \frac{p-q}{2} \rceil +q-1}\Big(-\alpha +k\delta\Big)  ~~~~~~~~~ \text{\Big [since, $(n+1)\binom{n}{r}=(r+1)\binom{n+1}{r+1}$\Big]}   \\
		=&\text{RHS}
	\end{split}\]
	
		\item if $x=(m-p,m+p), q+2\leq p\leq m$:
		\[\begin{split}
			\text{LHS}=&\displaystyle\binom{\lceil \frac{p-q+1}{2} \rceil+ q-1}{\,q\,}\prod_{k=\lceil \frac{p-q+1}{2} \rceil}^{\lceil \frac{p-q+1}{2} \rceil+q-1 }\Big(\alpha+k\delta \Big)\cdot  \big( \lceil\frac{p}{2}\rceil-\lceil\frac{q}{2}\rceil \big)\Big(\alpha+\big(\lceil\frac{p}{2}\rceil+\lceil\frac{q}{2}\rceil \big)\delta \Big)\\
			=&\displaystyle \lceil \frac{p-q-1}{2} \rceil\binom{\lceil \frac{p-q-1}{2} \rceil+q }{\,q\,}\prod_{k=\lceil \frac{p-q-1}{2} \rceil+1}^{\lceil \frac{p-q-1}{2} \rceil+q }\Big(\alpha+k\delta \Big)\cdot  \Big(\alpha+\big(\lceil \frac{p-q-1}{2} \rceil +q+1 \big)\delta \Big)       \text{ ~~[cf. Remark \ref{useful_rel_q_odd_ii}]} \\
			=&(q+1)\displaystyle	\binom{\lceil \frac{p-q-1}{2} \rceil + q}{\,q+1\,}\prod_{k=\lceil \frac{p-q-1}{2} \rceil +1}^{\lceil \frac{p-q-1}{2} \rceil +q+1}\Big(\alpha +k\delta\Big)  ~~~~~~~~~  ~~~~~ \text{\Big [since, $(n-r)\binom{n}{r}=(r+1)\binom{n}{r+1}$\Big]}   \\
			=&\text{RHS}
		\end{split}\]
\end{enumerate}
	\end{proof}

\begin{rema}\label{useful_rel_q_odd_ii}
	When $q$ is odd, the following equalities are applied in case (i), second line, of the proof of \eqref{mult:odd+-+-} and in case (ii), second line, of the proof of \eqref{mult:odd-+-+}.

		\[\lceil \frac{p-q+1}{2} \rceil-1 =\lceil \frac{p-q-1}{2} \rceil ; \qquad \lceil\frac{p}{2} \rceil+\lceil\frac{q}{2}\rceil=\lceil\frac{p-q-1}{2}\rceil+q+1,\qquad \lceil\frac{p}{2} \rceil-\lceil\frac{q}{2}\rceil=\lceil\frac{p-q-1}{2}\rceil.\] 
\end{rema}	

	\begin{rema}\label{useful_rel_q_odd_i}
		When $q$ is odd, the following equalities are applied in case (ii), second line, of the proof of \eqref{mult:odd+-+-} and in case (i), second line, of the proof of \eqref{mult:odd-+-+}.

		\[\lceil \frac{p-1}{2}\rceil- \lceil \frac{q}{2}\rceil+1=\lceil\frac{p-q}{2}\rceil,\qquad \lceil \frac{p-1}{2}\rceil+ \lceil \frac{q}{2}\rceil=\lceil\frac{p-q}{2}\rceil+q.\]
	\end{rema}

\begin{coro} \label{rep:deg4}
	We have the following relations:
	\begin{equation}\label{m+2,m-2}
		2\xi(m+2,m-2)=\xi(m-1,m+1)\Big(\xi(m-1,m+1)-p^{(m-1,m+1)}_{(m-1,m+1)}\Big)
	\end{equation}
	\begin{equation}\label{m-2,m+2}
		2\xi(m-2,m+2)=\xi(m+1,m-1)\Big(\xi(m+1,m-1)-p^{(m+1,m-1)}_{(m+1,m-1)}\Big)\end{equation}
\end{coro}

\begin{proof}
By setting $q=1$, \eqref{m+2,m-2} follows from \eqref{mult:odd+-+-}, and \eqref{m-2,m+2} follows from \eqref{mult:odd-+-+}.
\end{proof}

\begin{propo}\label{mult:q_even_propo}
For each even integer $q$ satisfying $0<q < m$, we have the following multiplicative relations:
	\begin{equation}\label{mult:even+--+}
		\xi(m+q,m-q)\Big(\xi(m-1,m+1)- p^{(m+q,m-q)}_{(m-1,m+1)}\Big)=0\cdot\xi(m+q+1,m-q-1)+(q+1)\cdot\xi(m-q-1,m+q+1) 
	\end{equation}
	\begin{equation}\label{mult:even-++-}
		\xi(m-q,m+q)\Big(\xi(m+1,m-1)-p^{(m-q,m+q)}_{(m+1,m-1)}\Big)=(q+1)\cdot\xi(m+q+1,m-q-1)+0\cdot\xi(m-q-1,m+q+1)
	\end{equation}
\end{propo}

	\begin{proof}
\underline{Step 1:	Equality of  \eqref{mult:even+--+}} 
	
	Note that the right-hand side of \eqref{mult:even+--+} vanishes when $x = (m+p, m-p)$ for $0 \leq p \leq q+1$, 
	and when $x = (m-p, m+p)$ for $0 \leq p \leq q$.  
	
	Furthermore, 	$\xi(m+q,m-q)(x)=0$ for $x = (m+p, m-p), 0 \leq p \leq q-1$, and  for $x = (m-p, m+p),0 \leq p \leq q$. 	Since $q$ is even,
	\[
	\xi(m-1, m+1)\big((m+q, m-q)\big)= p^{(m+q, m-q)}_{(m-1, m+1)}
	= \xi(m-1, m+1)\big((m+q+1, m-q-1)\big),
	\]
	it follows that the left-hand side of \eqref{mult:even+--+} also vanishes on each of these vertices. 
	
	Next, we verify  \eqref{mult:even+--+} by evaluating it on the remaining vertices $x \in \mathcal{V}$.

	\begin{enumerate}[(i)]
		
		\item if $x=(m+p,m-p), q+2\leq p\leq m$:
		\[\begin{split}
			\text{LHS}=&\binom{\lceil \frac{p-q+1}{2} \rceil + q-1}{\,q\,}\displaystyle\prod_{k=\lceil \frac{p-q+1}{2} \rceil -1}^{\lceil \frac{p-q+1}{2} \rceil +q-2}\Big(-\alpha +k\delta\Big)  \cdot    (\lceil \frac{p-1}{2}\rceil-\lceil \frac{q-1}{2}\rceil)    \Big(-\alpha +(\lceil \frac{p-1}{2}\rceil+\lceil \frac{q-1}{2}\rceil)\delta \Big)\\
			=& \displaystyle \lceil \frac{p-q-1}{2}\rceil \cdot\binom{\lceil \frac{p-q-1}{2} \rceil + q}{\,q\,}\prod_{k=\lceil \frac{p-q-1}{2} \rceil }^{\lceil \frac{p-q-1}{2} \rceil +q-1}\Big(-\alpha +k\delta\Big)  \cdot   \Big(-\alpha +(\lceil \frac{p-q-1}{2}\rceil+q)\delta \Big)  \text{ ~~[cf. Remark \ref{useful_rel_q_even_ii}]} \\
			=&(q+1)	\binom{\lceil \frac{p-q-1}{2} \rceil + q}{\,q+1\,}\displaystyle\prod_{\lceil \frac{p-q-1}{2} \rceil}^{\lceil \frac{p-q-1}{2} \rceil +q}\Big(-\alpha +k\delta\Big)~~~~~ \text{\Big [since, $(n-r)\binom{n}{r}=(r+1)\binom{n}{r+1}$\Big]}\\
			=&\text{RHS}
		\end{split}\]

		\item if $x=(m-p,m+p), q+1\leq p\leq m$:	
		\[\begin{split}
			\text{LHS}=& 		\displaystyle	\binom{\lceil \frac{p-q}{2} \rceil + q-1}{\,q\,}\prod_{k=\lceil \frac{p-q}{2} \rceil +1}^{\lceil \frac{p-q}{2} \rceil +q}\Big(\alpha +k\delta\Big)  \cdot  \Big(\lceil \frac{p}{2}\rceil+\lceil \frac{q-1}{2}\rceil\Big)\Big(\alpha+(\lceil \frac{p}{2}\rceil- \lceil \frac{q-1}{2}\rceil)\delta\Big)\\
			=&	\displaystyle\Big( \lceil\frac{p-q}{2}\rceil+q\Big)	\binom{\lceil \frac{p-q}{2} \rceil + q-1}{\,q\,}\prod_{k=\lceil \frac{p-q}{2} \rceil +1}^{\lceil \frac{p-q}{2} \rceil +q}\Big(\alpha +k\delta\Big)  \cdot  \Big(\alpha+\lceil \frac{p-q}{2}\rceil\delta\Big) \text{ ~~~~~~[cf. Remark \ref{useful_rel_q_even_i}]}  \\
			=&(q+1)	\displaystyle\binom{\lceil \frac{p-q}{2} \rceil+ q}{\,q+1\,}\prod_{k=\lceil \frac{p-q}{2} \rceil}^{\lceil \frac{p-q}{2} \rceil+q }\Big(\alpha+k\delta \Big)  ~~~~~ \text{\Big [since, $(n+1)\binom{n}{r}=(r+1)\binom{n+1}{r+1}$\Big]}  \\
			= &\text{RHS}
		\end{split}\]			
	\end{enumerate}
	
	\underline{Step 2:	Equality of  \eqref{mult:even-++-}} 
	
	Note that the right-hand side of \eqref{mult:even-++-}  vanishes when $x = (m+p, m-p)$ for $0 \leq p \leq q$, 
	and when $x = (m-p, m+p)$ for $0 \leq p \leq q+1$. 
	
	Furthermore, 	$\xi(m-q,m+q)(x)=0$ for $x = (m+p, m-p), 0 \leq p \leq q$, and  for $x = (m-p, m+p),0 \leq p \leq q-1$. Since $q$ is even,
	\[\xi(m+1, m-1)\big((m-q, m+q)\big)=p^{(m-q, m+q)}_{(m+1, m-1)}	= \xi(m+1, m-1)\big((m-q-1, m+q+1)\big).\]
	It follows that the left-hand side of \eqref{mult:even-++-} also vanishes on each of those vertices. 	Next, we verify \eqref{mult:even-++-} by evaluating it on the remaining vertices $x \in \mathcal{V}$.

	\begin{enumerate}[(i)]
		\item if $x=(m+p,m-p), q+1\leq p\leq m$:
		\[\begin{split}
			\text{LHS}=& \displaystyle	\binom{\lceil \frac{p-q}{2} \rceil + q-1}{\,q\,}\prod_{k=\lceil \frac{p-q}{2} \rceil}^{\lceil \frac{p-q}{2} \rceil +q-1}\Big(-\alpha +k\delta\Big)  \cdot\Big( \lceil \frac{p}{2} \rceil +\lceil \frac{q-1}{2} \rceil\Big) \big(-\alpha +\big(\lceil \frac{p}{2} \rceil -\lceil \frac{q-1}{2} \rceil -1\big)\delta\big)     \\
			=&  \displaystyle       \big(\lceil\frac{p-q}{2}\rceil+q\big)           \binom{\lceil \frac{p-q}{2} \rceil + q-1}{\,q\,}              \prod_{k=\lceil \frac{p-q}{2} \rceil -1}^{\lceil \frac{p-q}{2} \rceil +q-1}\Big(-\alpha +k\delta\Big)        \text{ ~~~~~~[cf. Remark \ref{useful_rel_q_even_i}]}  \\
			=&(q+1) \displaystyle	\binom{\lceil \frac{p-q}{2} \rceil + q}{\,q+1\,}\prod_{k=\lceil \frac{p-q}{2} \rceil -1}^{\lceil \frac{p-q}{2} \rceil +q-1}\Big(-\alpha +k\delta\Big)  ~~~~~~~~~ \text{\Big [since, $(n+1)\binom{n}{r}=(r+1)\binom{n+1}{r+1}$\Big]}   \\
			=&\text{RHS}
		\end{split}\]
		
		\item if $x=(m-p,m+p), q+2\leq p\leq m$:
		\[\begin{split}
			\text{LHS}=&\displaystyle\binom{\lceil \frac{p-q+1}{2} \rceil+ q-1}{\,q\,}\prod_{k=\lceil \frac{p-q+1}{2} \rceil}^{\lceil \frac{p-q+1}{2} \rceil+q-1 }\Big(\alpha+k\delta \Big)\cdot  \big( \lceil\frac{p-1}{2}\rceil-\lceil\frac{q-1}{2}\rceil \big)\Big(\alpha+\big(\lceil\frac{p-1}{2}\rceil+\lceil\frac{q-1}{2}\rceil +1\big)\delta \Big)\\
			=&\displaystyle \lceil \frac{p-q-1}{2} \rceil\binom{\lceil \frac{p-q-1}{2} \rceil+q }{\,q\,}\prod_{k=\lceil \frac{p-q-1}{2} \rceil+1}^{\lceil \frac{p-q-1}{2} \rceil+q }\Big(\alpha+k\delta \Big)\cdot  \Big(\alpha+\big(\lceil \frac{p-q-1}{2} \rceil +q+1 \big)\delta \Big)       \text{ ~~[cf. Remark \ref{useful_rel_q_even_ii}]} \\
			=&(q+1)\displaystyle	\binom{\lceil \frac{p-q-1}{2} \rceil + q}{\,q+1\,}\prod_{k=\lceil \frac{p-q-1}{2} \rceil +1}^{\lceil \frac{p-q-1}{2} \rceil +q+1}\Big(\alpha +k\delta\Big)  ~~~~~~~~~  ~~~~~ \text{\Big [since, $(n-r)\binom{n}{r}=(r+1)\binom{n}{r+1}$\Big]}   \\
			=&\text{RHS}
		\end{split}\] 	
	\end{enumerate}\end{proof}

\begin{rema}\label{useful_rel_q_even_ii} 
	When $q$ is even, the following equalities are applied in case (i), second line, of the proof of \eqref{mult:even+--+} and in case (ii), second line, of the proof of \eqref{mult:even-++-}.
	\[\lceil \frac{p-q+1}{2} \rceil-1 =\lceil \frac{p-q-1}{2} \rceil ; \qquad \lceil\frac{p-1}{2} \rceil+\lceil\frac{q-1}{2}\rceil=\lceil\frac{p-q-1}{2}\rceil+q,\qquad \lceil\frac{p-1}{2} \rceil-\lceil\frac{q-1}{2}\rceil=\lceil\frac{p-q-1}{2}\rceil.\] 
\end{rema}	

\begin{rema}\label{useful_rel_q_even_i} 
	When $q$ is even, the following equalities are applied in case (ii), second line, of the proof of \eqref{mult:even+--+} and in case (i), second line, of the proof of \eqref{mult:even-++-}.
	\[\lceil \frac{p}{2}\rceil- \lceil \frac{q-1}{2}\rceil=\lceil\frac{p-q}{2}\rceil,\qquad \lceil \frac{p}{2}\rceil+ \lceil \frac{q-1}{2}\rceil=\lceil\frac{p-q}{2}\rceil+q.\]
\end{rema}

\begin{proof}[Proof of Theorem \ref{gen_in_cohom-ring}]\label{abcd}
Any constant map (e.g. $\xi(m,m)$) trivially satisfies the congruence relation for  any edge in $\mathcal{G}_m$, and hence is an element of $H^*_T(X_m)$. 

Next, we show that $\xi(m+1,m-1)$ and $\xi(m-1,m+1)$ belong to $H^*_T(X_m)$. We verify this for $\xi(m+1,m-1)$ by checking, case by case, that it satisfies the congruence relation \eqref{def:eqcohom} for all edges in $\mathcal{G}_m$ (see \eqref{label:edge} for the edge labels). The argument for $\xi(m-1,m+1)$ is similar.


\begin{enumerate}[(i)]

\item if $x=(m+p,m-p)$ and $y=(m+l,m-l)$, $1\leq l<p\leq m$,  $p-l \text{ is odd}$:
	\[	\begin{split}
		\xi(m+1,m-1)(x)-\xi(m+1,m-1)(y) &=  (\lceil \frac{p}{2}\rceil-\lceil \frac{l}{2}\rceil)    \Big(-\alpha +(\lceil \frac{p}{2}\rceil+\lceil \frac{l}{2}\rceil-1)\delta \Big)\\
		&\equiv 0 \text{ mod } 	(p-l)\big(-\alpha+\frac{(p+l-1)}{2}\delta \big). 
		\end{split}\]
		
\item if $x=(m+p,m-p)$ and $y=(m,m)$, $1\leq p\leq m$,  $p \text{ is odd}$:
\[
	\xi(m+1,m-1)(x)-\xi(m+1,m-1)(y) =  \lceil \frac{p}{2}\rceil    \big(-\alpha +(\lceil \frac{p}{2}\rceil-1)\delta \big)
	\equiv 0 \text{ mod } 	p\big(-\alpha+\frac{(p-1)}{2}\delta \big).\]

\item if $x=(m+p,m-p)$ and $y=(m-1,m+1)$, $1\leq p\leq m$,  $p -1 \text{ is odd}$:
\[	\begin{split}
	\xi(m+1,m-1)(x)-\xi(m+1,m-1)(y) &=  \lceil \frac{p}{2}\rceil    \big(-\alpha +(\lceil \frac{p}{2}\rceil-1)\delta \big)
	\equiv 0 \text{ mod } 	(p+1)\big(-\alpha+\frac{(p-2)}{2}\delta \big).
\end{split}\]

\item if $x=(m+p,m-p)$ and $y=(m-l,m+l)$, $2\leq l<p\leq m$,  $p-l \text{ is odd}$:
	\[	\begin{split}
		\xi(m+1,m-1)(x)-\xi(m+1,m-1)(y) &=  (\lceil \frac{p}{2}\rceil+\lceil \frac{l-1}{2}\rceil)    \Big(-\alpha +(\lceil \frac{p}{2}\rceil-\lceil \frac{l-1}{2}\rceil-1)\delta \Big)\\
		&\equiv 0 \text{ mod } 	(p+l)\big(-\alpha+\frac{(p-l-1)}{2}\delta \big) 
	\end{split}\]

	\item if $x=(m-p,m+p)$ and $y=(m-l,m+l)$, $2\leq l<p\leq m$,  $p-l \text{ is odd}$:
	\[	\begin{split}
		\xi(m+1,m-1)(x)-\xi(m+1,m-1)(y) &=  (\lceil \frac{p-1}{2}\rceil-\lceil \frac{l-1}{2}\rceil)    \Big(\alpha +(\lceil \frac{p-1}{2}\rceil+\lceil \frac{l-1}{2}\rceil+1)\delta \Big)\\
		&\equiv 0 \text{ mod } 	(p-l)\big(\alpha+\frac{(p+l+1)}{2}\delta \big). 
	\end{split}\]

	\item if $x=(m-p,m+p)$ and $y=(m+l,m-l)$, $1\leq l<p\leq m$,  $p-l \text{ is odd}$:
		\[	\begin{split}
		\xi(m+1,m-1)(x)-\xi(m+1,m-1)(y) &=  (\lceil \frac{p-1}{2}\rceil-\lceil \frac{l}{2}\rceil)    \Big(\alpha +(\lceil \frac{p-1}{2}\rceil-\lceil \frac{l}{2}\rceil+1)\delta \Big)\\
		&\equiv 0 \text{ mod } 	(p+l)\big(\alpha+\frac{(p-l+1)}{2}\delta \big). 
	\end{split}\]
	
	\item if $x=(m-p,m+p)$ and $y=(m-1,m+1)$, $2\leq p\leq m$,  $p-1 \text{ is odd}$:
	\[
	\xi(m+1,m-1)(x)-\xi(m+1,m-1)(y) =  \lceil \frac{p-1}{2}\rceil    \big(\alpha +(\lceil \frac{p-1}{2}\rceil-1)\delta \big)
	\equiv 0 \text{ mod } 	(p-1)\big(\alpha+\frac{(p+2)}{2}\delta \big).\]
	
	\item if $x=(m-p,m+p)$ and $y=(m,m)$, $2\leq p\leq m$,  $p \text{ is odd}$:
	\[
\xi(m+1,m-1)(x)-\xi(m+1,m-1)(y) =  \lceil \frac{p-1}{2}\rceil    \big(\alpha +(\lceil \frac{p-1}{2}\rceil-1)\delta \big)
\equiv 0 \text{ mod } 	p\big(\alpha+\frac{(p+1)}{2}\delta \big).\]
\end{enumerate}
Since $H^*_T(X)$ has a ring structure and $\xi(m+1,m-1),~\xi(m-1,m+1)\in H^*_T(X)$, it follows from Corollary~\ref{rep:deg4} that $\xi(m+2,m-2)$ and $\xi(m-2,m+2)$ also belong to $H^*_T(X_m)$.

Proceeding inductively in the same manner, we conclude that $\xi(m+q,m-q)$ and $\xi(m-q,m+q)$ belong to $H^*_T(X_m)$ for all $1\leq q\leq m$, by using Propositions~\ref{mult:q_odd_propo} and~\ref{mult:q_even_propo}.
\end{proof}

The following propositions are not essential for our discussion, so their proofs are omitted to avoid unnecessary length and are left to the reader.
	\begin{propo} 
For each even integer $q$ satisfying $0 < q < m$, we have the following multiplicative relations:
	\begin{equation}
		\xi(m+q,m-q)\Big(\xi(m+1,m-1)- p^{(m+q,m-q)}_{(m+1,m-1)}\Big)=1\cdot\xi(m+q+1,m-q-1)+q\cdot\xi(m-q-1,m+q+1)
	\end{equation}
	\begin{equation}
		\xi(m-q,m+q)\Big(\xi(m-1,m+1)-p^{(m-q,m+q)}_{(m-1,m+1)}\Big)=q\cdot\xi(m+q+1,m-q-1)+1\cdot\xi(m-q-1,m+q+1)
	\end{equation}
\end{propo}

	\begin{propo}
For each odd integer $q$ satisfying $0 < q < m$, we have the following multiplicative relations:
	\begin{equation}
		\xi(m+q,m-q)\Big(\xi(m-1,m+1)- p^{(m+q,m-q)}_{(m-1,m+1)}\Big)=1\cdot\xi(m+q+1,m-q-1)+q\cdot\xi(m-q-1,m+q+1) \end{equation}
	\begin{equation}
	\xi(m-q,m+q)\Big(\xi(m+1,m-1)-p^{(m-q,m+q)}_{(m+1,m-1)}\Big)=q\cdot\xi(m+q+1,m-q-1)+1\cdot\xi(m-q-1,m+q+1)\end{equation}
\end{propo}

\begin{exam}
	We state some multiplicative relations among the Knutson--Tao basis elements for the $T$-action on $X(1,2,4)$ :
	
	\scalebox{0.9}{
\begin{tabular}{ll}
	$\xi(6,2)\cdot \big(\xi(3,5)	-p^{(6,2)}_{(3,5)}\big)=0\cdot\xi(7,1)+3\cdot\xi(1,7)$,~~~~~~~~~~~~~ & $\xi(2,6)\cdot\big( \xi(5,3)	-p^{(2,6)}_{(5,3)}\big)=3\cdot\xi(7,1)+0\cdot\xi(1,7)$,\\
	&\\
	$\xi(2,6)\cdot \big(\xi(3,5)	-p^{(2,6)}_{(3,5)}\big)=2\cdot\xi(7,1)+1\cdot\xi(1,7)$, &	$\xi(6,2)\cdot \big(\xi(5,3)	-p^{(6,2)}_{(5,3)}\big)=1\cdot\xi(7,1)+2\cdot\xi(1,7)$,\\
	&\\
	$\xi(1,7)\cdot\big(\xi(3,5)-p^{(1,7)}_{(3,5)}\big)=4\cdot\xi(8,0)+0\cdot\xi(0,8)$, &$\xi(7,1)\cdot\big(\xi(5,3)-p^{(7,1)}_{(5,3)}\big)=0\cdot\xi(8,0)+4\cdot\xi(0,8)$,\\
	&\\

	$\xi(7,1)\cdot\big(\xi(3,5)-p^{(7,1)}_{(3,5)}\big)=1\cdot\xi(8,0)+3\cdot\xi(0,8)$, &	$\xi(1,7)\cdot\big(\xi(5,3)-p^{(1,7)}_{(5,3)}\big)=3\cdot\xi(8,0)+1\cdot\xi(0,8)$,\\
	&\\
	$\xi(5,3)\cdot\big(\xi(5,3)-p^{(5,3)}_{(5,3)}\big)=2\cdot\xi(2,6)$, &$\xi(3,5)\cdot\big(\xi(3,5)-p^{(3,5)}_{(3,5)}\big)=2\cdot\xi(6,2)$,\\
	&\\
	$\xi(5,3)\cdot \xi(3,5)=\xi(6,2)+\xi(2,6)$.
\end{tabular}}
\end{exam}

\begin{exam}
We state some multiplicative relations among the Knutson--Tao basis elements  for the $T$-action on $X(1,2,5)$ :
			\[\begin{split}
				\xi(9,1)\cdot\big(\xi(4,6)-p^{(9,1)}_{(4,6)}\big)=0\xi(10,0)+5\xi(0,10) ~~~~	&\xi(9,1)\cdot\big(\xi(6,4)-p^{(7,3)}_{(6,4)}\big)=5\xi(10,0)+0\xi(0,10)\\
				&\\
				\xi(9,1)\cdot\big(\xi(6,4)-p^{(9,1)}_{(6,4)}\big)=1\xi(10,0)+4\xi(0,10) ~~~~	&\xi(9,1)\cdot\big(\xi(4,6)-p^{(7,3)}_{(4,6)}\big)=4\xi(10,0)+1\xi(0,10)\\
				&\\
				\xi(8,2)\cdot\big(\xi(6,4)-p^{(8,2)}_{(6,4)}\big)=0\xi(9,1)+4\xi(1,9) ~~~~~~~	&\xi(2,8)\cdot\big(\xi(4,6)-p^{(2,8)}_{(4,6)}\big)=4\xi(9,1)+0\xi(1,9)\\
				&\\
				\xi(8,2)\cdot\big(\xi(4,6)-p^{(8,2)}_{(4,6)}\big)=1\xi(9,1)+3\xi(1,9) ~~~~~~~	&\xi(2,8)\cdot\big(\xi(6,4)-p^{(2,8)}_{(6,4)}\big)=3\xi(9,1)+1\xi(1,9)\\
				&\\
				\xi(7,3)\cdot\big(\xi(4,6)-p^{(7,3)}_{(4,6)}\big)=0\xi(8,2)+3\xi(2,8) ~~~~~~~	&\xi(3,7)\cdot\big(\xi(6,4)-p^{(7,3)}_{(6,4)}\big)=3\xi(8,2)+0\xi(2,8)\\
				&\\
				\xi(7,3)\cdot\big(\xi(6,4)-p^{(7,3)}_{(6,4)}\big)=1\xi(8,2)+2\xi(2,8) ~~~~~~~	&\xi(3,7)\cdot\big(\xi(4,6)-p^{(7,3)}_{(4,6)}\big)=2\xi(8,2)+1\xi(2,8)\\
					&\\
				\xi(6,4)\cdot\big(\xi(6,4)-p^{(6,4)}_{(6,4)}\big)=0 \xi(7,3)+2\xi(3,7) ~~~~~~~	&\xi(4,6)\cdot\big(\xi(4,6)-p^{(4,6)}_{(4,6)}\big)=2\xi(7,3)+0 \xi(3,7)\\
					&\\
				\xi(6,4)\cdot\xi(4,6)=\xi(7,3)+\xi(3,7)~~~~~~~ &
			\end{split}\] 
\end{exam}


	\section{Degree two generators and defining relations.}\label{s:gen_rel}
In this section, we describe the ring structure of the equivariant cohomology of $X(1,2,m)$ under the $   (\mathbb{C}^*)^2$-action in terms of generators and relations. We observe that it suffices to consider two degree-$2$ Knutson--Tao basis elements, namely $\xi(m+1,m-1)$ and $\xi(m-1,m+1)$, as generators, together with three relations among them.

\begin{defe}[Degree-two generators]
	\label{def:degree_two_generators}
	We recall the degree-two elements from Definition~\ref{ktclasses} by setting $q=1$.
	\[
	\xi(m+1,m-1)(x)=
	\begin{cases}
		\left\lceil \dfrac{p}{2} \right\rceil \Big(-\alpha + \big(\left\lceil \dfrac{p}{2} \right\rceil -1\big)\delta\Big), & \text{if } x=(m+p,m-p),~1\leq p\leq m, \\[6pt]
		\left\lceil \dfrac{p-1}{2} \right\rceil \Big(\alpha + \big(\left\lceil \dfrac{p-1}{2} \right\rceil +1\big)\delta\Big), & \text{if } x=(m-p,m+p),~2\leq p\leq m, \\[6pt]
	~~~~~~~~~~~~	0, & \text{otherwise.}
	\end{cases}
	\]
	\[
	\xi(m-1,m+1)(x)=
	\begin{cases}
		\left\lceil \dfrac{p-1}{2} \right\rceil \Big(-\alpha + \left\lceil \dfrac{p-1}{2} \right\rceil \delta\Big), & \text{if } x=(m+p,m-p),~2\leq p\leq m, \\[6pt]
		\left\lceil \dfrac{p}{2} \right\rceil \Big(\alpha + \left\lceil \dfrac{p}{2} \right\rceil \delta\Big), & \text{if } x=(m-p,m+p),~1\leq p\leq m, \\[6pt]
	~~~~~~~~~~~~	0, & \text{otherwise.}
	\end{cases}
	\]
\end{defe}


	The following theorem shows that every Knutson--Tao basis element can be expressed in terms of the degree-two Knutson--Tao basis elements $\xi(m+1,m-1)$ and $\xi(m-1,m+1)$.
	
\begin{guess}\label{surj}
	Let $q$ be an integer such that $0 < q \leq m$. Then the following relations hold:
	
	\begin{enumerate}
		\item \underline{if $q$ is odd}
		\begin{equation}\label{surj:qodd_+-}
			\xi(m+q,m-q)=\displaystyle\frac{1}{ q!}\xi(m+1,m-1)\cdot\prod_{\substack{1\leq r< q ,\\ r~ \text{odd}}}\Big(\xi(m+1,m-1)-p^{(m+r,m-r)}_{(m+1,m-1)}\Big)\cdot\prod_{\substack{1\leq r< q,\\ r~ \text{even}}}\Big(\xi(m+1,m-1)-p^{(m-r,m+r)}_{(m+1,m-1)}\Big)
		\end{equation}
		\begin{equation}\label{surj:qodd-+}
			\xi(m-q,m+q)=\displaystyle\frac{1}{ q!}\xi(m-1,m+1)\cdot\prod_{\substack{1\leq r< q ,\\ r~ \text{odd}}}\Big(\xi(m-1,m+1)-p^{(m-r,m+r)}_{(m-1,m+1)}\Big)\cdot\prod_{\substack{1\leq r< q,\\ r~ \text{even}}}\Big(\xi(m-1,m+1)-p^{(m+r,m-r)}_{(m-1,m+1)}\Big)
		\end{equation}
		
		\item \underline{if $q$ is even}
		\begin{equation}\label{surj:qeven_+-}
			\xi(m+q,m-q)=\displaystyle\frac{1}{ q!} \xi(m-1,m+1)\cdot\prod_{\substack{1\leq r< q ,\\ r~ \text{odd}}}\Big(\xi(m-1,m+1)-p^{(m-r,m+r)}_{(m-1,m+1)}\Big)\cdot\prod_{\substack{1\leq r< q,\\ r~ \text{even}}}\Big(\xi(m-1,m+1)-p^{(m+r,m-r)}_{(m-1,m+1)}\Big)
		\end{equation}
		\begin{equation}\label{surj:qeven_-+}
			\xi(m-q,m+q)=\displaystyle\frac{1}{q!}\xi(m+1,m-1)\cdot\prod_{\substack{1\leq r< q ,\\ r~ \text{odd}}}\Big(\xi(m+1,m-1)-p^{(m+r,m-r)}_{(m+1,m-1)}\Big)\cdot\prod_{\substack{1\leq r< q,\\ r~ \text{even}}}\Big(\xi(m+1,m-1)-p^{(m-r,m+r)}_{(m+1,m-1)}\Big)
		\end{equation}
	\end{enumerate}
\end{guess}
\begin{proof}
	The result follows from Propositions~\ref{mult:q_odd_propo} and~\ref{mult:q_even_propo}. In particular, \eqref{surj:qodd_+-} is obtained by alternating applications of \eqref{mult:even-++-} and \eqref{mult:odd+-+-}, while \eqref{surj:qodd-+} follows from alternating applications of \eqref{mult:even+--+} and \eqref{mult:odd-+-+}.
	
	Similarly, \eqref{surj:qeven_+-} follows from alternating applications of \eqref{mult:odd-+-+} and \eqref{mult:even+--+}, and \eqref{surj:qeven_-+} from alternating applications of \eqref{mult:odd+-+-} and \eqref{mult:even-++-}.
\end{proof}

\begin{propo}\label{rel:1}
	We have the following relation among two degree 2 generators:
	\begin{equation}
		\begin{split}
			2\xi(m+1,m-1)\xi(m-1,m+1)=&\xi(m+1,m-1)\Big(\xi(m+1,m-1)-p^{(m+1,m-1)}_{(m+1,m-1)}\Big)\\&+\xi(m-1,m+1)\Big(\xi(m-1,m+1)-p^{(m-1,m+1)}_{(m-1,m+1)}\Big)
		\end{split}
	\end{equation}
\end{propo}

\begin{proof}
	The result directly follows from Corollary \ref{rep:deg4} and Proposition \ref{mult:deg2deg2}.
\end{proof}

\begin{propo}\label{rel:2}
	We have the following relations among two degree 2 generators:
	\begin{align}
&\displaystyle \xi(m+1,m-1)\cdot\prod_{\substack{1\leq r\leq m ,\\ r~ \text{odd}}}\Big(\xi(m+1,m-1)-p^{(m+r,m-r)}_{(m+1,m-1)}\Big)\cdot\prod_{\substack{1 \leq r\leq m,\\ r~ \text{even}}}\Big(\xi(m+1,m-1)-p^{(m-r,m+r)}_{(m+1,m-1)}\Big)=0 \label{4.39}\\
&\displaystyle \xi(m-1,m+1)\cdot\prod_{\substack{1\leq r\leq m ,\\ r~ \text{odd}}}\Big(\xi(m-1,m+1)-p^{(m-r,m+r)}_{(m-1,m+1)}\Big)\cdot\prod_{\substack{1\leq r\leq m,\\ r~ \text{even}}}\Big(\xi(m-1,m+1)-p^{(m+r,m-r)}_{(m-1,m+1)}\Big)=0 \label{4.40}
\end{align}
\end{propo}

\begin{proof}
We prove the result assuming that $m$ is even. The case when $m$ is odd can be treated similarly.

By setting $q=m$ in \eqref{surj:qeven_+-} and \eqref{surj:qeven_-+} of Theorem \ref{surj}, we have 
	\begin{equation}\label{proof:rep_2m0}
	\xi(2m,0)=\displaystyle\frac{1}{ m!} \xi(m-1,m+1)\cdot\prod_{\substack{1\leq r< m ,\\ r~ \text{odd}}}\Big(\xi(m-1,m+1)-p^{(m-r,m+r)}_{(m-1,m+1)}\Big)\cdot\prod_{\substack{1\leq r< m,\\ r~ \text{even}}}\Big(\xi(m-1,m+1)-p^{(m+r,m-r)}_{(m-1,m+1)}\Big),
\end{equation}
\begin{equation}\label{proof:rep_02m}
	\xi(0,2m)=\displaystyle\frac{1}{m!}\xi(m+1,m-1)\cdot\prod_{\substack{1\leq r< m ,\\ r~ \text{odd}}}\Big(\xi(m+1,m-1)-p^{(m+r,m-r)}_{(m+1,m-1)}\Big)\cdot\prod_{\substack{1\leq r< m,\\ r~ \text{even}}}\Big(\xi(m+1,m-1)-p^{(m-r,m+r)}_{(m+1,m-1)}\Big).
\end{equation}

Moreover, by Proposition \ref{mult:highest}, we have
\begin{equation}\label{proof:highestdeg_2m0}	\xi(2m,0)\Big(\xi(m-1,m+1)-p_{(m-1,m+1)}^{(2m,0)}\Big)=0 \end{equation}	
\begin{equation}\label{proof:highestdeg_02m} \xi(0,2m)\Big(\xi(m+1,m-1)-p_{(m+1,m-1)}^{(0,2m)}\Big)=0.
\end{equation}

Hence, \eqref{4.39} follows from \eqref{proof:highestdeg_2m0} and \eqref{proof:rep_2m0}, while  \eqref{4.40} follows from \eqref{proof:highestdeg_02m} and \eqref{proof:rep_02m}.
\end{proof}

Let $\bq[T][\xi^+_m,\xi^-_m]$ be the polynomial ring generated by $\xi^+_m$ and $\xi^-_m$ over the coefficient ring $\bq[T]=\bq[\alpha,\delta]$ \eqref{toruspresent}. 
Let $\mathcal{I}\triangleleft \bq[T][\xi^+_m,\xi^-_m]$ be an ideal generated by the following  elements:
\begin{enumerate}[(i)]
	\item $			2\xi^+_m\xi^-_m-\xi^+_m\Big(\xi^+_m-p^{(m+1,m-1)}_{(m+1,m-1)}\Big)-\xi^-_m\Big(\xi^-_m-p^{(m-1,m+1)}_{(m-1,m+1)}\Big)$,
	\item $ \xi^+_m\cdot\prod_{\substack{1\leq r\leq m ,\\ r~ \text{odd}}}\Big(\xi^+_m-p^{(m+r,m-r)}_{(m+1,m-1)}\Big)\cdot\prod_{\substack{1 \leq r\leq m,\\ r~ \text{even}}}\Big(\xi^+_m-p^{(m-r,m+r)}_{(m+1,m-1)}\Big)$,
\item $ \xi^-_m \cdot\prod_{\substack{1\leq r\leq m ,\\ r~ \text{odd}}}\Big(\xi^-_m -p^{(m-r,m+r)}_{(m-1,m+1)}\Big)\cdot\prod_{\substack{1\leq r\leq m,\\ r~ \text{even}}}\Big(\xi^-_m -p^{(m+r,m-r)}_{(m-1,m+1)}\Big)$.
\end{enumerate}

Let 
\[
\widetilde{\phi}: \mathbb{Q}[T][\xi^+_m,\xi^-_m] \to H^*_T(X(1,2,m))
\] 
be a ring homomorphism defined by
\[
\xi^+_m \mapsto \xi(m+1,m-1), \quad \xi^-_m \mapsto \xi(m-1,m+1).
\]

Let
\begin{equation}\label{eq:phi_induced}
	\phi: \mathbb{Q}[T][\xi^+_m,\xi^-_m] /\mathcal{I} \to H^*_T(X(1,2,m))
\end{equation}
be the homomorphism induced from $\widetilde{\phi}$. Note that the well-definedness of $\phi$ follows immediately from the Proposition \ref{rel:1} and Proposition \ref{rel:2}.

In other words, the following diagram commutes:
\[
\begin{tikzcd}[row sep=1cm, column sep=1cm]
	&  \mathbb{Q}[\alpha, \delta, \xi^+_m,\xi^-_m] \arrow[d,"pr"'] \arrow[dr, "\widetilde{\phi}"] & \\
	&\mathbb{Q}[\alpha, \delta, \xi^+_m,\xi^-_m] /\mathcal{I} \arrow[r, "\phi"'] & H^*_T(X(1,2,m))
\end{tikzcd}
\]
where the vertical map is the natural projection.

Now we state the main theorem of this paper.

\begin{guess}\label{theo:main_cqG}
	The homomorphism 
	\[
	\phi:\mathbb{Q}[\alpha, \delta, \xi^+_m,\xi^-_m] /\mathcal{I} \longrightarrow H^*_T(X(1,2,m))
	\] 
	is an isomorphism of rings.
\end{guess}

\begin{proof}
	From Proposition \ref{ktclassformsbasis}, it follows that the collection
	\[
	\Big\{\xi(m,m),~\xi(m+q,m-q),~\xi(m-q,m+q) : 1 \le q \le m\Big\}
	\]
	forms a $\mathbb{Q}[\alpha,\delta]$-module \eqref{toruspresent} basis of $H^*_T(X(1,2,m))$. Using Theorem \ref{surj}, the  generating set (as a $\mathbb{Q}[\alpha,\delta]$-algebra) further reduces to
	\[
	\{\xi(m+1,m-1),~\xi(m-1,m+1)\}.
	\]
	Hence, the surjectivity of $\phi$ follows.

	Moreover, from the definition of 
	\[
	\mathbb{Q}[\alpha,\delta, \xi^+_m,\xi^-_m] / \mathcal{I},
	\]
	one can check that
	\[
	\{1, (\overline{\xi^+_m})^i, (\overline{\xi^-_m})^j : 1 \le i,j \le m\}
	\]
	forms a $\mathbb{Q}[\alpha,\delta]$-basis.  Hence, as a $\mathbb{Q}[\alpha,\delta]$-module, the rank of $\mathbb{Q}[\alpha,\delta, \xi^+_m,\xi^-_m] / \mathcal{I}$ is $2m+1$, which matches the rank of $H^*_T(X(1,2,m))$. 
	
	Therefore, $\phi$ is injective, and the theorem follows.
\end{proof}

%
%
%


	From Definition~\ref{Def:p^_}, one can observe that for any $m_1, m_2 \in \bn$ and $1 \leq q \leq m_i$, $i = 1, 2$, we have  
	\[
	p^{(m_1+q,\,m_1-q)}_{(m_1+1,\,m_1-1)} = p^{(m_2+q,\,m_2-q)}_{(m_2+1,\,m_2-1)}, 
	\qquad 
	p^{(m_1-q,\,m_1+q)}_{(m_1+1,\,m_1-1)} = p^{(m_2-q,\,m_2+q)}_{(m_2+1,\,m_2-1)}.
	\]
	Therefore, for any $q > 0$, we define 
	\[
	p^{(q,-q)}_{+} := p^{(2q,0)}_{(q+1,\,q-1)}, 
	\qquad 
	p^{(-q,q)}_{+} := p^{(0,2q)}_{(q+1,\,q-1)},
	\]
	and similarly, 
	\[
	p^{(q,-q)}_{-} := p^{(2q,0)}_{(q-1,\,q+1)}, 
	\qquad 
	p^{(-q,q)}_{-} := p^{(0,2q)}_{(q-1,\,q+1)}.
	\]

	Let $\bq[T][\xi^+,\xi^-]$ be the polynomial ring generated by $\xi^+$ and $\xi^-$ over the coefficient ring $\bq[T]=\bq[\alpha,\delta]$ \eqref{toruspresent}.   
	Let $\mathcal{J} \triangleleft \bq[T][\xi^+,\xi^-]$ be the ideal generated by the following elements:
	\begin{enumerate}[(i)]
		\item $2\xi^+\xi^- - \xi^+\big(\xi^+ - p^{(1,-1)}_{+}\big) - \xi^-\big(\xi^- - p^{(-1,1)}_{-}\big)$,
		\item $\xi^+ \cdot \displaystyle\prod_{\substack{q \ge 1\\ q~\text{odd}}}\!\big(\xi^+ - p^{(q,-q)}_{+}\big) 
		\cdot \!\!\!\!\!\prod_{\substack{q \ge 1\\ q~\text{even}}}\!\big(\xi^+ - p^{(-q,q)}_{+}\big)$,
		\item $\xi^- \cdot \displaystyle\prod_{\substack{q \ge 1\\ q~\text{odd}}}\!\big(\xi^- - p^{(-q,q)}_{-}\big) 
		\cdot \!\!\!\!\!\prod_{\substack{q \ge 1\\ q~\text{even}}}\!\big(\xi^- - p^{(q,-q)}_{-}\big)$.
	\end{enumerate}

\begin{coro}\label{theo:main_afg}
	The homomorphism 
	\[
	\phi: \mathbb{Q}[\alpha, \delta, \xi^+,\xi^-] / \mathcal{J} \longrightarrow H^*_T(\mathscr{AF}_2)
	\]
	is an isomorphism of rings.
\end{coro}

\begin{proof}
	Using~\eqref{eq:afv_in_terms_of_qG}, we have
	\[
	H^*_T(\mathscr{AF}_2) \cong \varprojlim_m H^*_T(X(1,2,m)).
	\]
	Therefore, the statement follows from Theorem~\ref{theo:main_cqG}.
\end{proof}


\section{Integrality of structure constants  }\label{s:int_str_cons}
In this section, we look at the structure constants of the algebra $H^*_T(X(1,2,m))$ using the Knutson--Tao type basis we defined (cf. Definition \ref{ktclasses}) earlier. We also study their integrality (cf. Theorem \ref{Theo:strconsts}), that is, we show that these constants belong to $\mathbb{Z}[T]=\bz[\alpha, \delta]$.

\begin{lema}\label{Lem:strconsts}
		Let $f \in H^*_T(X(1,2,m))$. Then $f$ can be written as 
		\begin{equation}\label{5.46}
		f = h_0 \, \xi(m,m) + \sum_{q=1}^{m} \big( h_{q,1} \cdot \xi(m+q,m-q) + h_{q,2} \cdot \xi(m-q,m+q) \big),
		\end{equation}
		where $h_0 = f(m,m)$, and for each $1 \le q \le m$,
		\[
		h_{q,1} = 
		\frac{\Big[f - h_0 \xi(m,m) - \sum_{r=1}^{q-1} 
			\big(h_{r,1} \xi(m+r,m-r) + h_{r,2} \xi(m-r,m+r)\big)\Big](m+q,m-q)}
		{p^{(m+q,m-q)}_{(m+q,m-q)}},
		\]
		\[
		h_{q,2} = 
		\frac{\big[f - h_0 \xi(m,m) - \sum_{r=1}^{q-1} 
			\Big(h_{r,1} \xi(m+r,m-r) + h_{r,2} \xi(m-r,m+r)\big)\Big](m-q,m+q)}
		{p^{(m-q,m+q)}_{(m-q,m+q)}}.
		\]
		Moreover, $h_0,\, h_{q,1},\, h_{q,2} \in \bq[\alpha, \delta]$ \eqref{toruspresent} for each $1 \le q \le m$.
\end{lema}

	\begin{proof}
		For any \( f \in H_T^*(X(1,2,m)) \), define \( h_0 := f(m,m) \). 
		Then, by \eqref{def:eqcohom}, we have \( h_0 \in \mathbb{Q}[\alpha, \delta] \).
		
		Define  
		\[
		f_{1} := f - h_0 \, \xi(m,m) \in H_T^*(X(1,2,m)).
		\]
		Note that \( f_{1}(m,m) = 0 \).  
		By considering the congruence relation on the edge \((m+1,m-1) \to (m,m)\), we obtain  
		\[
		f_{1}(m+1,m-1) - f_{1}(m,m) \equiv 0 
		\pmod{\alpha\big((m+1,m-1)\to (m,m)\big)}.
		\]
		
		From the definition of \( p^{(m+1,m-1)}_{(m+1,m-1)} \) (cf. Remark \ref{prod_of_edge_label}), it follows that  
		\[
		f_{1}(m+1,m-1) = h_{1,1} \, p^{(m+1,m-1)}_{(m+1,m-1)} 
		\quad \text{for some } h_{1,1} \in \mathbb{Q}[\alpha,\delta].
		\]
		Similarly, we have  
		\[
		f_{1}(m-1,m+1) = h_{1,2} \, p^{(m-1,m+1)}_{(m-1,m+1)} 
		\quad \text{for some } h_{1,2} \in \mathbb{Q}[\alpha,\delta].
		\]
		
		Define  
		\[
		\begin{split}
			f_{2} 
			&:= f_1 - h_{1,1} \, \xi(m+1,m-1) - h_{1,2} \, \xi(m-1,m+1) \\
			&= f - h_0 \, \xi(m,m) - h_{1,1} \, \xi(m+1,m-1) - h_{1,2} \, \xi(m-1,m+1)
			\in H_T^*(X(1,2,m)).
		\end{split}
		\]
		Note that \( f_{2}(m,m) = f_{2}(m+1,m-1) = f_{2}(m-1,m+1) = 0 \).
		
		By considering the congruence relations on the edges 
		\((m+2,m-2) \to (m+1,m-1)\) and \((m+2,m-2) \to (m-1,m+1)\), 
		we obtain respectively  
		\begin{align*}
			f_{2}(m+2,m-2) - f_{2}(m+1,m-1)
			&\equiv 0 \pmod{\alpha\big((m+2,m-2)\to (m+1,m-1)\big)}, \\
			f_{2}(m+2,m-2) - f_{2}(m-1,m+1)
			&\equiv 0 \pmod{\alpha\big((m+2,m-2)\to (m-1,m+1)\big)}.
		\end{align*}
		
		Again, from the definition of \( p^{(m+2,m-2)}_{(m+2,m-2)} \) (cf. Remark \ref{prod_of_edge_label}), 
		we can write  
		\[
		f_{2}(m+2,m-2) = h_{2,1} \, p^{(m+2,m-2)}_{(m+2,m-2)} 
		\quad \text{for some } h_{2,1} \in \mathbb{Q}[\alpha,\delta].
		\]
		Similarly,  
		\[
		f_{2}(m-2,m+2) = h_{2,2} \, p^{(m-2,m+2)}_{(m-2,m+2)} 
		\quad \text{for some } h_{2,2} \in \mathbb{Q}[\alpha,\delta].
		\]
		
		Continuing in this way, for \( 3 \leq q \leq m-1 \), we assume  
		\[
		f_{q}(m+q,m-q) = h_{q,1} \, p^{(m+q,m-q)}_{(m+q,m-q)} 
		\quad \text{for some } h_{q,1} \in \mathbb{Q}[\alpha,\delta],
		\]
		and similarly,  
		\[
		f_{q}(m-q,m+q) = h_{q,2} \, p^{(m-q,m+q)}_{(m-q,m+q)} 
		\quad \text{for some } h_{q,2} \in \mathbb{Q}[\alpha,\delta].
		\]
		
		Define 
		\[
		\begin{split}
			f_{q+1} 
			&:= f_q - h_{q,1} \, \xi(m+q,m-q) - h_{q,2} \, \xi(m-q,m+q) \\
			&= f - h_0 \, \xi(m,m) 
			- \sum_{r=1}^{q} h_{r,1} \, \xi(m+r,m-r)
			- \sum_{r=1}^{q} h_{r,2} \, \xi(m-r,m+r)
			\in H_T^*(X(1,2,m)).
		\end{split}
		\]
		Note that 
		\[
		f_{q+1}(m+r,m-r) = 0 = f_{q+1}(m-r,m+r)
		\quad \text{for each } 0 \leq r \leq q.
		\]
		
		By considering the congruence relations on the edges 
		\((m+q+1,m-q-1) \to (m+r,m-r)\) and 
		\((m+q+1,m-q-1) \to (m-r,m+r)\),
		where \(0 \leq r \leq q\) and \(q+1-r\) is odd, 
		we obtain respectively  
		\begin{align*}
			f_{q+1}(m+q+1,m-q-1) - f_{q+1}(m+r,m-r)
			&\equiv 0 \pmod{\alpha\big((m+q+1,m-q-1)\to (m+r,m-r)\big)}, \\
			f_{q+1}(m+q+1,m-q-1) - f_{q+1}(m-r,m+r)
			&\equiv 0 \pmod{\alpha\big((m+q+1,m-q-1)\to (m-r,m+r)\big)}.
		\end{align*}
		
		From the definition of \( p^{(m+q+1,m-q-1)}_{(m+q+1,m-q-1)} \) (cf. Remark \ref{prod_of_edge_label}),
		we can write  
		\[
		f_{q+1}(m+q+1,m-q-1) 
		= h_{q+1,1} \, p^{(m+q+1,m-q-1)}_{(m+q+1,m-q-1)} 
		\quad \text{for some } h_{q+1,1} \in \mathbb{Q}[\alpha,\delta],
		\]
		and similarly,  
		\[
		f_{q+1}(m-q-1,m+q+1) 
		= h_{q+1,2} \, p^{(m-q-1,m+q+1)}_{(m-q-1,m+q+1)} 
		\quad \text{for some } h_{q+1,2} \in \mathbb{Q}[\alpha,\delta].
		\]
		
		Finally, define 
		\[
		\begin{split}
			f_{m+1}
			&:= f_m - h_{m,1} \, \xi(2m,0) - h_{m,2} \, \xi(0,2m) \\
			&= f - h_0 \, \xi(m,m)
			- \sum_{r=1}^{m} h_{r,1} \, \xi(m+r,m-r)
			- \sum_{r=1}^{m} h_{r,2} \, \xi(m-r,m+r)
			\in H_T^*(X(1,2,m)).
		\end{split}
		\]
		Here,
		\[
		f_{m+1}(m+r,m-r) = 0 = f_{m+1}(m-r,m+r)
		\quad \text{for all } r = 0,1,\ldots,m.
		\]
		Hence $f_{m+1} \equiv 0$, i.e. 
		\begin{equation}
			f = h_0 \, \xi(m,m) + \sum_{r=1}^{m} \big( h_{r,1} \cdot \xi(m+r,m-r) + h_{r,2} \cdot \xi(m-r,m+r) \big),
		\end{equation}
		Therefore, the lemma follows.
	\end{proof}
	\begin{rema}\label{ktclass_indep}
The above lemma, in particular, implies that  
\begin{equation}\label{basis_kt_type}
	\Big\{\xi(m,m),~\xi(m+q,m-q),~\xi(m-q,m+q) : 1 \le q \le m\Big\}
\end{equation}
spans $H_T^*(X(1,2,m))$. Moreover, one can verify that the set in \eqref{basis_kt_type} is linearly independent, and hence forms a $\bq[T]$-module basis of $H_T^*(X(1,2,m))$.

		Indeed, suppose that $f = 0$ in \eqref{5.46}. Evaluating 
		\[
		h_0\,\xi(m,m) + \sum_{q=1}^{m} \big( h_{q,1}\,\xi(m+q,m-q) + h_{q,2}\,\xi(m-q,m+q) \big)=0
		\]
		at the vertices of the moment graph of $X(1,2,m)$, starting from $(m,m)$ and proceeding to $(m+1,m-1)$, $(m-1,m+1)$, and so on, 
		one observes from Definition~\ref{ktclasses} that 
		\[h_0 = h_{q,1} = h_{q,2} = 0 \text{ for all $q = 1, \ldots, m$.}\]
	\end{rema}

	\begin{guess}\label{Theo:strconsts}
		For arbitrary Knutson--Tao basis elements 
		$\xi(2m-i,i)$ and $\xi(2m-j,j)$ with $0 \leq i,j \leq 2m$, 
		the product can be expressed as
		\[
	\xi(2m-i,i)\cdot \xi(2m-j,j)= h_0 \, \xi(m,m) 
		+ \sum_{q=1}^{m} \big( h_{q,1} \, \xi(m+q,m-q) 
		+ h_{q,2} \, \xi(m-q,m+q) \big),
		\]
		where $ h_0,\, h_{q,1},\, h_{q,2} \in \mathbb{Z}[\alpha,\delta] $ for each $ 1 \le r \le m $.
	\end{guess}
	
\begin{proof}
	Consider the product 
	\[	f = \xi(2m-i,i) \cdot \xi(2m-j,j).	\]
	From Definition~\ref{ktclasses}, we have \( f(2m-l,l) \in \mathbb{Z}[\alpha,\delta] \) for each \( 0 \le l \le 2m \).
	
	We now proceed by induction on \( q \).  
	By Definition~\ref{ktclasses}, it follows that \( h_0 = f(m,m) \in \mathbb{Z}[\alpha,\delta] \).  
	Moreover, by Lemma~\ref{Lem:strconsts}, we obtain
	\[
	h_{1,1} = 
	\frac{\big[f - h_0 \, \xi(m,m)\big](m+1,m-1)}
	{p^{(m+1,m-1)}_{(m+1,m-1)}} \in \mathbb{Q}[\alpha,\delta],
	\qquad
	h_{1,2} = 
	\frac{\big[f - h_0 \, \xi(m,m)\big](m-1,m+1)}
	{p^{(m-1,m+1)}_{(m-1,m+1)}} \in \mathbb{Q}[\alpha,\delta].
	\]
	
	By Definitions~\ref{Def:p^_} and~\ref{ktclasses}, both 
	\( p^{(m+1,m-1)}_{(m+1,m-1)} \) and \( p^{(m-1,m+1)}_{(m-1,m+1)} \) 
	are products of linear factors whose coefficients lie in \( \{\pm 1\} \),
	while the corresponding numerators belong to \( \mathbb{Z}[\alpha,\delta] \).  
	Hence, \( h_{1,1},\, h_{1,2} \in \mathbb{Z}[\alpha,\delta] \).
	
	Now assume that for some \( q \le m \), we have 
	\( h_{r,1},\, h_{r,2} \in \mathbb{Z}[\alpha,\delta] \) for all \( 1 \le r \le q-1 \).
	By Lemma~\ref{Lem:strconsts}, it follows that
	\( h_{q,1},\, h_{q,2} \in \mathbb{Q}[\alpha,\delta] \).
	Again, by Definitions~\ref{Def:p^_} and~\ref{ktclasses}, 
	\( p^{(m+q,m-q)}_{(m+q,m-q)} \) and 
	\( p^{(m-q,m+q)}_{(m-q,m+q)} \)
	are products of linear factors with coefficients in \( \{\pm 1\} \),
	whereas the numerators lie in \( \mathbb{Z}[\alpha,\delta] \), since
	\( f(m+q,m-q),\, f(m-q,m+q) \in \mathbb{Z}[\alpha,\delta] \)
	and \( h_{r,1},\, h_{r,2} \in \mathbb{Z}[\alpha,\delta] \) for all \( 1 \le r \le q-1 \).
	Therefore, \( h_{q,1},\, h_{q,2} \in \mathbb{Z}[\alpha,\delta] \).
	
	By induction, the claim holds for all \( 1 \le q \le m \).
	Hence, the theorem follows.
\end{proof}

{\bf Acknowledgement:} 
The author expresses sincere gratitude to Professor Evgeny Feigin for introducing him to quiver Grassmannians, and for his constant encouragement,  insightful discussions, and valuable suggestions throughout this work. This research was supported by a postdoctoral fellowship funded through Professor Feigin’s ISF grant 493/24.





\end{document}